\documentclass[10pt]{amsart}
\usepackage{amssymb}
\usepackage{listings}
\newcommand{\Ps}{\mathbf{P}}
\newcommand{\C}{\mathbf{C}}
\newcommand{\Q}{\mathbf{Q}}
\newcommand{\Z}{\mathbf{Z}}
\newcommand{\G}{\mathbf{G}}
\newcommand{\cL}{\mathcal{L}}
\newcommand{\cO}{\mathcal{O}}
\newcommand{\cD}{\mathcal{D}}

\newtheorem{lemma}{Lemma}[section]
\newtheorem{proposition}[lemma]{Proposition}
\newtheorem{theorem}[lemma]{Theorem}
\newtheorem{corollary}[lemma]{Corollary}
\newtheorem{conjecture}[lemma]{Conjecture}
\theoremstyle{definition}
\newtheorem{notation}[lemma]{Notation}
\newtheorem{definition}[lemma]{Definition}
\newtheorem{construction}[lemma]{Construction}
\newtheorem{example}[lemma]{Example}

\theoremstyle{remark}

\newtheorem{remark}[lemma]{Remark}
\DeclareMathOperator{\CH}{CH}
\DeclareMathOperator{\NL}{NL}
\DeclareMathOperator{\red}{red}
\DeclareMathOperator{\prim}{prim}
\DeclareMathOperator{\codim}{codim}
\DeclareMathOperator{\rank}{rank}

\title[Hodge loci of linear subvarieties]{On a conjecture on Hodge loci of  linear combinations of linear subvarieties}
\author{Remke Kloosterman}
\email{klooster@math.unipd.it}
\address{Universit\`a degli Studi di Padova,
Dipartimento di Matematica ``Tullio Levi-Civita",
Via Trieste 63,
35121 Padova, Italy}
\thanks{The author would like to thank Hossein Movasati for inviting him to IMPA in 2019 and for several comments on a previous version of this paper. The author would also like to thank Roberto Villaflor Loyola  for several comments on a previous version of this paper.
The author is a member of PRIN-2020 Project ``Curves, Ricci flat Varieties and their Interactions" and of INdAM-GNSAGA. This work was supported  by the BIRD-SID-2021 UniPD
project
 ``Moduli spaces of curves and hypersurface singularities".}
\date{\today}
\begin{document}
\begin{abstract}
For each $k \geq 5$ we give a counterexample to a conjecture of Movasati on the dimension of certain Hodge loci of cubic hypersurfaces in $\Ps^{2k+1}$ containing two $k$-planes intersecting in dimension $k-3$. We give similar examples for  Hodge loci of cubic hypersurfaces in $\Ps^{2k+1}$ containing two $k$-planes intersecting in dimension $k-2$ and for quartic  hypersurfaces in $\Ps^{2k+1}$ containing two $k$-planes intersecting in dimension $k-2$.

Moreover, we present new evidence for Movasati's conjecture for  the values of $k$ for which our type of counterexamples cannot exist, i.e., for $k=3,4$.
\end{abstract}
\maketitle

\section{Introduction}\label{secInt}
Fix  integers $k\geq 1,d\geq 2+\frac{2}{k}$.  Let $U_d\subset \Ps(\C[x_0,\dots,x_{2k+1}]_d)$ be the space of smooth hypersurfaces of degree $d$ in $\Ps^{2k+1}$. Let $[X]\in U_d$ be a smooth hypersurface containing a non-trivial Hodge class. Let  $\gamma\in H^{2k}(X,\Q)\cap H^{k,k}(X)$ be such a Hodge class, i.e.,  $\gamma_{\prim}$ is  nonzero. 
In a simply connected neighborhood of $[X]$ in $U_d$ one can define the Hodge locus $\NL(\gamma)$  as the locus where $\gamma$ remains of type $(k,k)$ under parallel transport. This locus has a natural analytic scheme structure, which turns out to be algebraic. 
Similarly, if we pick several Hodge classes $\gamma_1,\dots,\gamma_s$ then we can define $\NL(\gamma_1,\dots,\gamma_s)$ as the locus where all $\gamma_i$ remain of type $(k,k)$. For more details see Section~\ref{secIdealCon} and the references therein.

Suppose now that $[X]\in U_d$ contains two distinct $k$-planes $\Pi_1,\Pi_2$. Let $c=k-\dim \Pi_1\cap \Pi_2$. A question posed in recent work by Movasati \cite{MovConb} is whether the equality  $\NL(a[\Pi_1]+b[\Pi_2])=\NL([\Pi_1],[\Pi_2])$ holds as analytic schemes. If $ab=0$ then $\NL(a[\Pi_1]+b[\Pi_2])\neq \NL([\Pi_1],[\Pi_2])$ holds. However, if $ab\neq 0$ then the answer depends on the choice of $(c,d,k)$.
A result by Villaflor \cite{RVLTs} implies that if both $(d-2)(c-1)>2$  and $ab\neq 0$ hold then $\NL(a[\Pi_1]+b[\Pi_2])=\NL([\Pi_1],[\Pi_2])$ as analytic schemes. 

We will now focus on the cases  not covered by Villaflor. The case $c=1$ is covered in the paper \cite{KloCodimOne}. In this paper we will consider the remaining cases where both $k-m>1$ and $(d-2)(c-1)\leq 2$ hold, i.e., $(d,c) \in \{(4,2),(3,3),(3,2)\}$.
Movasati \cite[Chapters 18 and 19]{MovBook} posted several conjectures on the irreducible components of the  Hodge locus $\NL(a[\Pi_1]+b[\Pi_2])$ where $a,b\in \Z$ such that $ab\neq 0$.
 Movasati focused mainly on  the case $d=3,c=3$:

\begin{conjecture}[{Movasati, \cite[Conjecture 19.1(2)]{MovBook}}] \label{ConMovIntro} Let $k\geq 2$ be an integer.
Suppose $X$ is a smooth hypersurface of degree $3$ in $\Ps^{2k+1}$ containing two $k$-planes such that their intersection has dimension $k-3$. Let $a,b$ be integers with  $ab\neq 0 $ then $\codim_{U_d} \NL(a[\Pi_1]+b[\Pi_2])=\codim_{U_d} \NL([\Pi_1],[\Pi_2])-1$.
\end{conjecture}
Movasati states that he expects  a similar result  for $d=4, c=2$ \cite[Section 19.8]{MovBook}, and that the corresponding Hodge loci for the case $d=3,c=2$ are singular \cite[Secion 19.2]{MovBook}. Movasati based his conjectures on extensive computer based  approximations of the local structure of these Hodge loci at the Fermat cubic hypersurface, and had to restrict to low values for $k$. In particular, for $k\geq 5$ he could only do approximations of low order.
We will analyse the same Hodge loci, but in the neighborhood of more general hypersurfaces. 
However, differently from Movasati's approach, we manage only to analyse the tangent space and did not succeed in higher order approximations.

Since a nonzero class is a Hodge class if and only if every nonzero multiple of this class is a Hodge class we may identify the locus $\NL(a[\Pi_1]+b\Pi)$ with $\NL([\Pi_1]+\frac{b}{a}[\Pi_2])$, provided that $a\neq 0.$ We prefer the latter notation, in order to simplify some of the statements.

Our first main result is the following:


\begin{theorem}\label{mainThmIntro} Suppose $d=3, c\in \{2,3\} ,k\geq 5$ or $d=4,c=2,k\geq 3$. Then there exists a smooth hypersurface $X\subset\Ps^{2k+1}$ of degree $d$ containing two $k$-planes intersecting in codimension $c$ and a finite set $S$, such that for  all $\lambda\in\Q\setminus S$ we have
\[ \codim_{T_{[X]}U_d} T_X\NL([\Pi_1]+\lambda [\Pi_2])=\codim_{T_{[X]}U_d} T_X \NL([\Pi_1],[\Pi_2]).\]
Moreover, $\NL([\Pi_1]+\lambda[\Pi_2])=\NL([\Pi_1],[\Pi_2])$ in a neighborhood of $X$.

Moreover, if $c=d=3$ then $S=\{0,1\}$ and if $d=4,c=2$ then $S=\{0,-1\}$.
\end{theorem}

In the final part of the paper we consider the case where $\lambda$ is the nonzero element of $S$. We show that Movasati's conjecture holds if $(c,d,\lambda)=(2,4,-1)$ and $k\geq 1$ or $(c,d,\lambda)=(3,3,1)$ and $k\geq 2$, see Corollary~\ref{corMovQua} and Proposition~\ref{corMovCub}. This settles also \cite[Conjecture 18.1]{MovBook} for $d=4$.

Our approach uses tangent space considerations and brakes down for some small values of $k$, and for these values our method actually provides evidence for Movasati's conjecture:
\begin{theorem}\label{thmMovConvIntro} Suppose $d=3, c\in \{2,3\} ,k\in \{2,3,4\} $ or $d=4,c=2,k\in \{1,2\}$. Suppose $X\subset\Ps^{2k+1}$ is a smooth hypersurface of degree $d$ containing two $k$-planes intersecting in codimension $c$. Then for all $\lambda\in\Q^*$ we have
\[ \codim T_X\NL([\Pi_1]+\lambda [\Pi_2])<\codim T_X \NL([\Pi_1],[\Pi_2]).\]
Moreover, for $d=c=3$, $k\in\{2,3,4\}$ and $\lambda \in \Q^*$, either Movasati's conjecture holds or $\NL([\Pi_1]+\lambda[\Pi_2])$ is nonreduced.
\end{theorem}

We would like to point out that Movasati's calculations for the Fermat hypersurface are mostly within the range of this theorem. As is noted in \cite[Section 5.10]{MovConb}, for $(d,c)=(3,3)$ there is a difference between the dimension of the tangent spaces to $\NL([\Pi_1],[\Pi_2])$ and $\NL([\Pi_1]+\lambda[\Pi_2])$ at the Fermat hypersurface for any $k\geq 2$. We will give the following explanation for this:
In the sequel we give a somewhat ad-hoc definition of split hypersurface of codimension $c$. One easily checks that the Fermat hypersurface is split of codimension $c$.  
\begin{theorem}\label{thmSplitIntro} Let $(c,d)\in \{(2,3),(3,3),(2,4)\}$. Let $k\geq c-1$ be an integer.
Let $X\subset \Ps^{2k+1}$ be a smooth hypersurface of degree $d$ containing two linear spaces  $\Pi_1,\Pi_2$ of dimension $k$, intersecting in dimension $k-c$ and of split type of codimension $c$. 
 Let $\lambda\in \Q^*$.  
Then $\NL([\Pi_1],[\Pi_2])_{\red}\subsetneq \NL([\Pi_1]+\lambda[\Pi_2])_{\red}$ in a neighbourhood of $X$.
Moreover, if $d=3$ and $k\geq 5$ or $d=4$ and $k\geq 3$ then $\NL([\Pi_1]+\lambda[\Pi_2])$ is reducible at $X$.
\end{theorem}

We will now give some details on our construction of the counterexamples.  
The first result is a result on the dimension of the Hodge locus of a cycle which is a combination of two  complete intersection cycles:
\begin{theorem}\label{thmTanSpIntro} Let $c,k,d_0,\dots, d_k, e_0,\dots,e_{c-1}$ be integers.
Let $Z\subset \Ps^{2k+1}$ be a complete intersection of multidegree $(d_0,\dots,d_k,e_0,\dots,e_{c-1})$, let $Y_1$ and $Y_2$ be complete intersections of multidegree $(d_0,\dots, d_k)$ and $(e_0,\dots,e_{c-1},d_c,\dots,d_k)$ respectively, such that their intersection equals $Z$. 
 If  
 \[ \sum_{i=0}^{c-1} (e_i+d_i)<(c-1)d\]
 then for all but finitely many $\lambda\in\Q^*$ we have
 \[ \codim T_X\NL([Y_1]+\lambda [Y_2])=\codim T_X \NL([Y_1],[Y_2]).\]
Moreover, $\NL([Y_1]+\lambda [Y_2])=\NL([Y_1])\cap \NL([Y_2])$.
\end{theorem}
In the case where $Y_1$ and $Y_2$ are linear subvarieties the inequality in the above Theorem simplifies to 
$(c-1)(d-2)>2$.
Hence the above result is a generalization of Villaflor's result \cite[Theorem 1.3]{RVLTs}.

To study the case $(c-1)(d-2)\leq 2$ we use the standard construction of an Artinian Gorenstein algebra associated with a Hodge class (Section~\ref{secIdealCon}). In the case of a complete intersection cycle $g_0=\dots=g_k=0$ we have that $f=\sum_{i=0}^k g_ih_i$ and the associated ideal is $\langle g_0,\dots,g_k,h_0,\dots,h_k\rangle $. The quotient ring $S/I$ is a graded Artinian Gorenstein algebra of  socle degree $(d-2)(k+1)$.

\begin{theorem}\label{thmPairingIntro}
Suppose $X\subset \Ps^{2k+1}$ is a smooth hypersurface of degree $d$ such that $X$ contains two subvarieties $Y_1,Y_2$ of dimension $k$. Suppose that $[Y_1]_{\prim}$ and $[Y_2]_{\prim}$ are linearly independent in $H^{2k}(X,\Q)_{\prim}$.

Let $I_i$ be the ideal associated with the Hodge class $[Y_i]$, cf. Construction~\ref{conIdeal}.
Suppose that the left kernel of the multiplication map
\[ (I_1+I_2/I_2)_d\times (I_1+I_2/I_2)_{kd-2k-2}\to (S/I_2)_{(k+1)(d-2)}\]
is zero. Then for all but finitely many $\lambda \in\Q^*$  we have
\[ \codim T_X\NL([Y_1]+\lambda [Y_2])=\codim T_X \NL([Y_1],[Y_2])\]
\end{theorem}

Consider now the case where $Y_1$ and $Y_2$ are $k$-planes such that $c=k-\dim Y_1\cap Y_2$.
We will show that if $(d-2)(c-1)>2$ holds then we cannot have a left kernel  for dimension reasons. On the other hand, 
if $d=3$, $c\in \{2,3\}$  and $k\in\{2,3,4\}$ hold then for dimension reasons we will always have a left kernel, and in fact we can show that the equality does not hold (see Example~\ref{exaRestriction}), hence either $\NL([Y_1]+\lambda[Y_2])$ is nonreduced or Movasati's conjecture holds. If $d=4, c=2$ and $k\in \{1,2\}$ then we are in a similar situation. This is the key observation in the proof of Theorem~\ref{thmMovConvIntro}.

However, for the same choices of $(c,d)$, but for larger $k$, we have that the dimension of the kernel depends on $X$, e.g.,  if $X$ is the Fermat hypersurface then we know that the conclusion of the above theorem does not hold, but on a Zariski open subset of $\NL([Y_1],[Y_2])$ we do not have a left kernel and therefore  $\NL([Y_1],[Y_2])=\NL(a[Y_1]+b[Y_2])$ if $X$ is general on $\NL([\Pi_1],[\Pi_2])$. The increase of the dimension of the tangent at the Fermat point can then be explained by the fact that $\NL(a[Y_1]+b[Y_2])$ is reducible, see Theorem~\ref{thmSplit}.

The organization of this paper is as follows:

In Section~\ref{secPrelim} we recall some results on Artinian Gorenstein algebras and bilinear maps used in the sequel. In Section~\ref{secIdealCon} we discuss the construction of an Artinian Gorestein algebra associated to a Hodge class on a hypersurface in projective space and we introduce the notation of excess tangent dimension. Moreover, we give a criterion to check whether there is excess tangent dimension (cf. Theorem~\ref{thmPairingIntro}). In Section~\ref{secMov} we construct the counterexamples to Movasati's conjecture. In Section~\ref{secSplit} we discuss the split case and give an explanation for the calculations obtained by Villaflor~\cite{RVLTs} and Movasati~\cite{MovBook}. In Section~\ref{secExaPos} we give a proof for Movasati's conjecture for the single value of $\lambda$ where our counterexample breaks down.
In Appendix~\ref{secCode} we give the computerscripts we used in our computation.
\section{Preliminaries}\label{secPrelim}

In this section let $S=\C[x_0,\dots,x_n]$ be the polynomial ring in $n+1$ variables, which we consider as a graded ring, with its usual grading.

\subsection{Artinian Gorenstein algebras}
In the sequel we rely many times on the following construction of Artinian Gorenstein algebras:
\begin{lemma}\label{lemGorensteinConstr}

Let $t$ be a positive integer.
Let $W\subset S_t$ be a linear subspace. Let $I\subset S$ be the largest ideal generated in degree $\leq t$ such that $I_t=W$. 

Suppose $W$ is base point free.
Then there exists a $\codim_{S_t} W-1$ family of graded Artinian Gorenstein algebras of socle degree $t$, each of which is a quotient of $S/I$.
Moreover, if $\codim_{S_t}W=1$ then $S/I$ is a graded Aritinian Gorenstein algebra with socle degree $t$.
\end{lemma}
\begin{proof}
We start by proving the Moreover-part:
If $\codim_{S_t} W=1$ then $S/I$ is an Artinian Gorenstein algebra: Let $e\leq t-1$  and $f\in S_e$  such that $f(S/I)_{t-e} $ is zero in $(S/I)_t$. Let $I'=I+\langle f\rangle$ then  $I'_t=I_t=W$ and $I\subset I'$. Hence our assumption on $I$ implies $I=I'$ and hence that $f\in I'$. Therefore the pairing $(S/I)_e\times (S/I)_{t-e}\to (S/I)_t$ is perfect.
Moreover, since $W$ is base point free we have by, e.g., \cite[Page 297]{GreenF} that $(S/I)_{t+1}=0$, and therefore $I_k=S_k$ for every $k\geq t+1$. In particular, $S/I$ has finite vector space dimension and therefore Krull dimension zero.

If $\codim_{S_t} W>1$ then we can apply this construction to the inverse imagine in $S_t$ of any codimension one subspace of $(S/I)_t$.
\end{proof}

\begin{example}\label{exaCI} Fix positive integers $d_0,\dots,d_n$. For $i=0,\dots,n$ let $f_i\in S_{d_i}$, such  that $f_0,\dots,f_n$ is a complete intersection.
Let $I=\langle f_0,\dots f_n \rangle$. Let $t=\sum_{i=0}^n d_i-n-1$. Then $\dim (S/I)_t=1$ and $\dim (S/I)_{t+1}=0$. This is actually a Artinian Gorenstein ring by a result of Macaulay. \cite[Corollary 21.19]{EisCA}
\end{example}

Let $I$ be a homogeneous ideal in $S$. Then we denote with $h_I:\Z \to \Z$ the Hilbert function of this ideal, i.e., $h_I(\alpha)=\dim_\C (S/I)_{\alpha}$.

Recall that for homogeneous ideals $I,J$ in a graded ring we have for every $\alpha\in \Z$ tthe following exact sequence of $\C$-vector spaces:
\[ 0 \to (I\cap J)_\alpha \to I_\alpha\oplus J_\alpha \to (I+J)_{\alpha}\to 0.\]
The following lemma follows immediately from the existence of this sequence:
\begin{lemma}\label{lemHP}
Let $I,J\subset S$ be homogeneous ideals such that  for $e$ sufficiently large we have that $(I+J)_e$ is base point free. Then for every $\alpha\in \Z$ we have
\[  h_{I\cap J}(\alpha)=h_I(\alpha)+h_J(\alpha)-h_{I+J}(\alpha).\]
Let $t$ be  an integer $t$ such $(I+J)_t\neq S_t$ and $(I+J)_{t+1}=S_{t+1}$ then for $\alpha\geq t+1$ we have
\[  h_{I\cap J}(\alpha)=h_I(\alpha)+h_J(\alpha).\]
\end{lemma}

\subsection{Complete intersection ideals}\label{subsecCII}
We will recall some results on complete intersection ideals.
Fix integers $c,m$ with $c\leq m$. Take $g_1,\dots,g_{m+c}\in S$ homogeneous polynomials which form a complete intersection.
Let $I=\langle g_1,\dots,g_m\rangle$ and let $J=\langle g_{c+1},\dots g_{m+c} \rangle.$ Then $\codim I=\codim J=m$, $\codim I+J=m+c$ and $I,J$ and $I+J$ are complete intersection ideals.

\begin{lemma} Let $g_1,\dots,g_{m+c},I,J$ as above. Suppose $f\in I\cap J$. Then there exists polynomials $Q_{ij},P_k$ for $1\leq i ,j\leq c; c+1\leq k \leq m$, such that
\[f= \sum_{i=1}^{c}\sum_{j=1}^{c} g_ig_{m+j} Q_{ij} +\sum_{i=c+1}^m g_iP_i.\]
\end{lemma}
\begin{proof}
Recall that  $I+J$ is a complete intersection ideal. This implies that the Koszul complex on $g_1,\dots,g_{m+c}$ is a minimal resolution of $I+J$. Denote with $e_1,\dots,e_{m+c}$ the standard basis of $R^{m+c}$. Consider the sequence of $R$-modules
\[ \wedge^2 R^{m+c}\to R^{m+c} \to R\]
where the first map sends $e_i\wedge e_j$ to $g_j e_i-g_ie_j$, the second map sends $e_i$ to $g_i$.
 Then this sequence is exact and the cokernel on the right equals $R/(I+J)$.

Since $f\in I \cap J$ there exists polynomials $a_i, b_i$, for $1\leq i \leq m$ such that
\[ f=\sum_{j=1}^m a_jg_j=\sum_{j=c+1}^{m+c} b_j g_j\]
This implies that
\[ \sum_{j=1}^c a_jg_j+\sum_{j=c+1}^m (a_j-b_j)g_j-\sum_{j=1}^c b_j g_{m+j} =0.\]
I.e., we obtained a syzygy among the minimal generators of $I+J$.

By the above observation on the exactness of the Koszul complex there exist  for $1\leq i<j \leq m+c$ homogeneous elements $k_{ij} \in S$  such that
for $t\leq c$ we have
\[ a_t=\sum_{i=1}^{t-1} k_{it}g_i-\sum_{i=t+1}^{m+c} k_{ti} g_i,\]
for $c+1\leq t\leq  m$ we have 
\[ a_t-b_t=\sum_{i=1}^{t-1} k_{it}g_i-\sum_{i=t+1}^{m+c} k_{ti} g_i,\]
and for $m+1\leq t \leq m+c$ we have
\[ -b_t=\sum_{i=1}^{t-1} k_{it} g_i-\sum_{i=t+1}^{m+c} k_{ti} g_{i}.\]

Obviously these $a_i$ are not unique in general. For $i=1,\dots, c$ define 
\[\tilde{a}_i:=a_i-\sum_{j=1}^{i-1} k_{ji} g_j+\sum_{j=i+1}^c k_{ij} g_j.\]
Then $\sum_{i=1}^c a_i g_i=\sum_{i=1}^c \tilde{a}_ig_i$ and 
\[ \tilde{a}_i =-\sum_{i=m+1}^{m+c} k_{ti} g_i \in J.\]
For $i=1,\dots, c$ write  $\tilde{a}_i=a'_i+a''_i$, with $a'_i\in I\cap J$ and $a''_i\in \langle g_{c+1},\dots,g_m \rangle$.
Then 
\begin{eqnarray*} f=\sum_{i=1}^m a_ig_j=\sum_{i=1}^c \tilde{a}_ig_i+\sim_{i=c+1}^m a_ig_i &\equiv &\sum_{i=1}^c a_i g_i \bmod  \langle g_{c+1},\dots, g_m \rangle\\
 &\equiv& \sum_{i=1}^c a'_i g_i \bmod \langle g_{c+1},\dots, g_m \rangle\end{eqnarray*}
In particular, there exists $P_i$ and $Q_{ij}$  such that
\[ f=\sum_{i=1}^c\sum_{j=1}^c Q_{ij} g_ig_{m+j}+\sum_{j=c+1}^m P_j g_j.\]
\end{proof}

\begin{remark}\label{rmkCIIdeal}
 In the case where $n$ is even and $m=n/2$ we will introduce the ideal associated to a Hodge class. For a hypersurface $X=V(f)$, where $X$ contains two $m-1$-dimensional complete intersections whose intersection is a $m-c-1$-dimensional complete intersection, we have two such ideals, namely
 \[\left\langle g_1,\dots,g_{c},\left(\sum_{j=1}^{c} h_jQ_{ij}\right)_{i=1}^{c},g_{c+1},\dots,g_{m},P_{c+1},\dots, P_{m}\right\rangle\]
 and
 \[ \left\langle h_1,\dots,h_{c},\left(\sum_{i=1}^{c} g_jQ_{ij}\right)_{j=1}^{c},g_{c+1},\dots,g_{m},P_{c+1},\dots, P_{m}\right\rangle.\]
 Their sum equals
 \[ \left\langle g_1,\dots,g_{c},h_1,\dots,h_{c},g_{c+1},\dots,g_{m},P_{c+1},\dots, P_{m}\right\rangle,\] which is again a complete intersection ideal, 
and their intersection equals
 \[ \left\langle (g_ih_j)_{i,j=1}^{c},(\sum_{j=1}^{c} h_jQ_{ij})_{i=1}^{c},(\sum_{i=1}^{c} g_jQ_{ij})_{j=1}^{c},g_{c+1},\dots,g_{m},P_{c+1},\dots, P_{m}\right\rangle.\] 
\end{remark}
\subsection{Bilinear maps}
We will recall some well-known results on bilinear maps. We were not able to identify a place in the literature containing the following results in the precise form we will need them.

For the rest of this section we use the following notation:
\begin{notation}
Let $V$ and $W$ be $\C$-vector spaces. 
For $i=1,2$, let $\varphi_i:V\times W \to \C$  be bilinear maps. Denote with $\ker_L(\varphi_i)$ the kernel of the induced map $V\to W^*$ sending $v$ to $w\mapsto \varphi_i(v,w)$ and, similarly, denote with $\ker_R(\varphi_i)$ the kernel of the induced map $W\to V^*$.

Let $V_i=\ker_L(\varphi_i),W_i=\ker_R(\varphi_i)$. Suppose that $V_1\cap V_2=0$ and $W_1\cap W_2=0$.
Let $r_i=\rank \varphi_i$ and let $s_i=\rank \varphi_j|_{V_i\times W_i}$
\end{notation}

\begin{lemma}\label{lemgenrk} With the above notation, we have for $i=1,2$
\[\max\{\rank \varphi_1+t\varphi_2\colon t\in \C^*\}\geq r_i+s_i.\]
\end{lemma}
\begin{proof} It suffices to prove this  statement for $i=1$.
We prove the statement by induction on $s_1$. If $s_1=0$ then there is nothing to prove.
Suppose now $s_1>0$. 
Pick $v\in V_1,w\in W_1$, such that $\varphi_2(v,w)\neq 0$. Let $V'$ be the orthogonal space of $w$ with respect to $\varphi_2$, let $W'$ be the orthogonal space of $v$ with respect to $\varphi_2$.
Then $V=\langle v \rangle \oplus V'$ and $W=\langle w \rangle \oplus W'$ are simultaneously orthogonal decompositions with respect to $\varphi_1,\varphi_2$, i.e., $\varphi_i(v,W')=0$ and $\varphi_i(V',w)=0$ hold for $i=1,2$.

Let $\psi_i$ be the restriction of $\varphi_i$ to $V'\times W'$. Then $\rank(\psi_1)=r_1$. Moreover, $\ker_L(\psi_1)=\ker_L(\varphi_1)\cap V'$ and $\ker_R(\psi_1)=\ker_R(\varphi_1)\cap W'$. The rank of $\psi_2$ on this space $s_1-1$.
Moreover the rank of $\varphi_1+t\varphi_2$ on $\langle v\rangle\times \langle w\rangle$ is one for all nonzero $t$.
I.e., by induction we have
\[\max\{\rank \varphi_1+t\varphi_2\colon t\in \C^*\}\geq 1+\max\{\rank \psi_1+t\psi_2\colon t\in \C^*\}=r_1+s_1.
\]
\end{proof}

\begin{lemma}\label{lemSpectrumSize} Using the above notation, suppose that $s_1=\dim V-r_1$ holds, i.e.,  $\ker_L \varphi_2|_{V_1\times W_1}=0$, then for at most $\dim V-s_1-s_2$ nonzero values of $t$ we have the strict inequality
$ \rank \varphi_1+t\varphi_2< \dim V$.
\end{lemma}
\begin{proof}
We prove this statement by a double induction on $s_1, s_2$. Suppose $s_1=0$ and $s_2=0$.
Pick bases for $V$ and $W$. Let $A_i$ be the Gram matrix of $\varphi_i$. The values of $t$ for which the rank of $A_1+tA_2$ is less than $\dim V$ are precisely the common zeros of all maximal minors of $A_1+tA_2$, which are polynomials in $t$ of degree at most $\dim V$.
Since the maximal rank is $\dim V$ there is at least one such a minor which is nonzero, which has degree at most $\dim V$. Therefore there are at most $\dim V$ common zeroes of these maximal minors.

Suppose now $s_1=0$ and $s_2>0$ then
we can pick $(v,w)\in V_2\times W_2$ such that $\varphi_1(v,w)\neq 0$. Let $V'$ be the orthogonal space of $w$ with respect to $\varphi_1$, let $W'$ be the orthogonal space of $v$ with respect to $\varphi_1$.
This is a simultaneous orthogonal decomposition.

Then for all nonzero $t$ we have that $\rank \varphi_1+t\varphi_2=1+\rank (\psi_1+t\psi_2)$.
By induction there are at most $(\dim V-1)-(s_2-1)$ values of $t$ where the rank of $\psi_1+t\psi_2$ is less than $\dim V-1$.

Suppose now $s_1>0$ and $s_2>0$ then
we can pick $(v,w)\in V_1\times W_1$ such that $\varphi_1(v,w)\neq 0$. Let $V'$ be the orthogonal space of $w$ with respect to $\varphi_2$, let $W'$ be the orthogonal space of $v$ with respect to $\varphi_2$.
This is a simultaneous orthogonal decomposition.
Then for all nonzero $t$ we have that the rank of $\varphi_1+t\varphi_2$ equals $1+\rank (\psi_1+t\psi_2)$.
By induction there are at most $(\dim V-1)-(s_1-1)-s_2$ values of $t$ where the rank of $\psi_1+t\psi_2$ is less than $\dim V-1$.
\end{proof}

\begin{remark}\label{remb01}
We will apply this lemma  mostly in cases where $\dim V=\dim W$; $r_1=r_2=(\dim V+1)/2$ and $s_1=s_2$ hold. To apply this lemma we need  moreover that $s_1=(\dim V-1)/2$.
In this case $\dim V-s_1-s_2=1$, i.e., there is at most one nonzero value of $t$ with rank drop.
\end{remark}

\begin{lemma}\label{lemKernelBnd}
Using the above notation we have that for $\{i,j\}=\{1,2\}$
\[ \dim \ker_L \varphi_j|_{V_i\times W_i} \leq r_1+r_2-\dim W.\]
\end{lemma}

\begin{proof} Without loss of generality we may assume that $i=1,j=2$.
Let $a=r_1+r_2-\dim V$, let $b=r_1+r_2-\dim W$. Then $\dim V_1=\dim V-r_1=r_2-a$. Similar formulae hold for $\dim V_2,\dim W_1$ and $\dim W_2$.

Let $W'$ be a subspace of $W$ containing $W_1$ and such that $W'\oplus W_2=W$.
Then $\varphi_2|_{V_1\times W}$ has no left kernel, and therefore its rank equals to the dimension of $V_1$, which equals $r_2-a$.
Since $W_1\cap W_2=0$ it follows that the rank of $\varphi_2|_{V_1\times W'}$ is also $r_2-a$. Now $\dim W'=\dim W-\dim W_2=(r_1+r_2-b)-(r_1-b)=r_2$.

Hence the rank of $\varphi_2|_{V_1\times W_1}$ is at least $r_2-a-(\dim W'-\dim W_1)$, which equals $r_2-b-a$.
Therefore the left kernel has dimension at most $\dim V_1-(r_2-b-a)=b$.
\end{proof}

We are now interested in a particular class of bilinear forms. Let $I_1,I_2\subset S$ be ideals such that for $j=1,2$ the algebra $S/I_j$ is a graded Artinian Gorenstein algebra. Assume that the socle degrees of $S/I_1$ and $S/I_2$ coincide, say their socle degree equals $t$.
Fix isomorphisms $\sigma_j:(S/I_j)_t\to \C$.

For $\alpha\in \Z$,
consider $\varphi_j: S/(I_1\cap I_2)_\alpha\times S/(I_1\cap I_2)_{t-\alpha} \to (S/I_j)_t\to \C$, the composition of the natural multiplication map  with $\sigma_j$.

\begin{lemma}\label{lemHighDeg}
If $h_{I_1+I_2}(t-\alpha)=0$ then for all $\lambda\in \C^*$ we have $\ker_L(\varphi_1+\lambda \varphi_2)=0$.
\end{lemma}

\begin{proof}
If $h_{I_1+I_2}(t-\alpha)=0$ then
\[ h_{I_1\cap I_2}(t-\alpha)=h_{I_1}(t-\alpha)+h_{I_2}(t-\alpha)\]
by Lemma~\ref{lemHP}.
The  left kernel of $\varphi_j$ equals $(I_j/I_1\cap I_2)_\alpha$. Hence the rank of $\varphi_j$ equals
$ h_{I_j}(\alpha)$, which equals $h_{I_j}(t-\alpha)$ by Gorenstein duality.
In particular, we have that $\dim S/(I_1\cap I_2)_{t-\alpha}=r_1+r_2$ and $\ker_R(\varphi_1)\oplus \ker_R(\varphi_2)=(S/I_1\cap I_2)_{t-\alpha}$.

Suppose now that $v\in \ker_L(\varphi_1+\lambda\varphi_2)$ and that $\lambda\neq0$. Then for all $w\in \ker_R(\varphi_2)$ we have
\[0= \varphi_1(v,w)+\lambda \varphi_2(v,w)=\varphi_1(v,w)\]
This implies that $v$ is in the ortogonal complement of $\ker_R(\varphi_2)$ with respect to $\varphi_1$, since it is also in the orthogonal space of $\ker_R(\varphi_1)$ we obtain that
$v$ is in the orthogonal space of $\ker_R(\varphi_1)+\ker_R(\varphi_2)=W$ with respect to $\varphi_1$, i.e., $v\in \ker_L(\varphi_1)$. Similarly $v\in \ker_L(\varphi_2)$ and therefore $v\in V_1\cap V_2=\{0\}$. 
\end{proof}

\begin{lemma}\label{lemDegShift} Let $I,J\subset S$ be  homogeneous ideal such that $S/I$ is Artinian Gorenstein. Let $\alpha$ be an integer such $0\leq \alpha \leq t-1$. Suppose the left kernel of
\[ \mu_{\alpha}:(J+I/I)_\alpha\times (J+I/I)_{t-\alpha} \to (S/I)_t\]
is nonzero. Then the left kernel of
\[ \mu_{\alpha+1}:(J+I/I)_{\alpha+1}\times (J+I/I)_{t-\alpha-1} \to (S/I)_t\]
is also nonzero.
\end{lemma}
\begin{proof}
Suppose $f$ is a nonzero element of the left kernel of $\mu_\alpha$. Since $\alpha<t$ we can find a linear form $\ell$ such that $\ell f$ is nonzero. (If $\ell f$ is zero for all $\ell \in S_1$ then $\deg(\ell f)=t+1$, contradicting our choice of $\alpha$.)

For any $g\in (J+I)_{t-\alpha-1}$ we have that $\ell g\in (J+I)_{t-\alpha}$ and therefore
\[ \mu_{\alpha+1}(\ell f,g)=\mu_{\alpha}(f,\ell g)=0\]
Hence $\ell f\in \ker_L(\mu_{\alpha+1})$.
\end{proof}

\section{Construction of the ideal of a Hodge class}\label{secIdealCon}
Fix  integers $k\geq 1, d\geq 2+\frac{2}{k}$. Let $S=\C[x_0,\dots,x_{2k+1}]$.
Let $X$ be a smooth hypersurface of degree $d$ in $\Ps^{2k+1}$, containing at least one Hodge class $\gamma$, such that $\gamma_{\prim}$ is nonzero.
Let $J$ be the Jacobian ideal of $X$.

\begin{construction}\label{conIdeal}
Let $\gamma \in H^{k,k} (X,\C)\cap H^{2k}(X,\Q)$ be a Hodge class such that $\gamma_{\prim}$ is not zero, i.e., $\gamma$ is not a multiple of the class associated with the intersection of $X$ with $k$ hyperplanes.

 Griffiths' work \cite{GriRat} on the period map yields  an identification of $H^{2k-p,p}(X)_{\prim}$ with $(S/J)_{(p+1)d-2k-2}$.
Moreover, by Carlson--Griffiths \cite{CarGri} there is a choice of  isomorphisms $H^{2k,2k}(X)\cong \C\cong (S/J)_{(d-2) (2k+2)}$ such that
 the cupproduct \[ H^{2k-p,p}(X)_{\prim} \times H^{p,2k-p}(X)_{\prim} \to H^{2k,2k}(X)\] equals the multiplication map
\[  (S/J)_{(p+1)d-2k-2} \times (S/J)_{(2k-p+1)d-2k-2} \to (S/J)_{(d-2)(2k+2)}.\]
Let $f_\gamma \in (S/J)_{(d-2)(k+1)}$ be  the image of $\gamma_{\prim}$. Consider now the subspace  $f_\gamma^\perp$ in $(S/J)_{(d-2)(k+1)}$. Let $W$ be its inverse image in $S_{(d-2)(k+1)}$ under the projection map. 
Then $J_{(d-2)(k+1)}\subset W$.  Since $J_{d-1}$ is base point-free and $(d-2)(k+1)\geq d-1$ we have that $W$ is also base point free.
Let $I(\gamma)$ be the largest ideal of $S$ generated in degree $\leq (k+1)(d-2)$  such that $I_{(k+1)(d-2)}=W$. 

Then $I(\gamma)$ is the \emph{ideal associated} with the Hodge class $\gamma$.
\end{construction}
\begin{remark}\label{rmkSocDeg}
By construction the algebra $S/I(\gamma)$ is an Artinian Gorenstein algebra of socle degree $(d-2)(k+1)$, see Lemma~\ref{lemGorensteinConstr}.
\end{remark}
\begin{definition}
Let $f\in S_d$ be an irreducible polynomial such that $X=V(f)$ is a smooth hypersurface of degree $d$. Then $[f]\in U_d$.

Suppose $X$ is on the Hodge locus, i.e., there is a $\gamma \in H^{k,k} (X,\C)\cap H^{2k}(X,\Q)$ be a Hodge class such that $\gamma_{\prim}$ is nonzero.

Let $U\subset U_d$ be analytic open subset containing $[f]$. Then $H^{2k}(X_u,\Z), u\in U$ defines a natural local system of free $\Z$-modules on $U$. Consider the holomorphic vector bundle $\mathcal{H}=H^{2k}(X_u,\Z)\otimes \mathcal{O}_U$ on $U$.
There is a natural decreasing filtration $F^i \mathcal{H}$  for $0\leq i \leq 2k$, induced by the Hodge filtration on $H^2(X_t,\Z)$, satisfying the transversality property
\[ \nabla(F^p \mathcal{H})\subset F^{p-1} \mathcal{H}\otimes \Omega^1_U\]
If $U$ is simply connected then we can extend $\gamma\in H^{2k}(X,\Q)$ naturally to a section of $H^{2k}(X_u,\Q)\to U$. Denote with $\gamma_u$ the value of this section  in $H^{2k}(X_u,\Q)$.
We define the Hodge locus of $\gamma$ in $U$ denoted by $\NL(\gamma)$,  as the locus
\[ \{u\in U\mid \gamma_u \in F^k \mathcal{H}_u\}\]
\end{definition}
\begin{remark} The Hodge locus is defined as the vanish locus of the induced section $\overline{\gamma}$ of $\mathcal{H}/F^k\mathcal{H}$, hence it has a natural structure of a complex analytic scheme. Cattani, Deligne and Kaplan show that this locus is actually an algebraic scheme. \cite{DelKap}
\end{remark}

\begin{remark}
The definition of the Hodge locus depends on the choice of the extension of $\gamma$ as a section of $H^2(X_u,\Z)$, hence it is only well-defined if $U$ is simply connected. Moreover, the intersection of $U$ with the Zariski closure of $\NL(\gamma)$ in $U_d$ may be strictly larger than $\NL(\gamma)$.
\end{remark}

\begin{example}\label{exaIdealCI}
Suppose now that $X$ contains a $k$-dimensional subvariety $Y$, which is a complete intersection in $\Ps^{2k+1}$. Let $g_0,\dots,g_k$ be a system of homogeneous generators of $I(Y)$.
Then there exists $h_0,\dots,h_k$ such that
\[ f=g_0h_0+\dots +g_kh_k.\]
In this case the ideal associated with $[Y]$ is
\[ \langle g_0,\dots,g_k,h_0,\dots,h_k\rangle.\]
See, e.g., \cite{KloCI}.
\end{example}

\begin{lemma}\label{lemCodimNL} Suppose that $k\geq 2$ or that $k=1$ and $d\neq 4$. Then we have
\[ \codim T_X\NL(\gamma) = \codim_{S_d} (I(\gamma))_d\]
\end{lemma}
\begin{proof}
Recall that  we have the standard exact sequence
\[ 0 \to \Theta_{X}\to \Theta_{\Ps^{2k+1}\mid X} \to N_{X\mid \Ps^{2k+1}} \to 0.\]
Using Bott's vanishing theorem it follows easily that $H^1(\Theta_{ \Ps^{2k+1}\mid X})=0$ under our assumption. Consider now 
\[  H^0(N_{X\mid \Ps^{2k+1}}) \twoheadrightarrow  H^1(X,\Theta_X) \cong (S/J)_{d},\]
where the  isomorphism follows from  \cite[Lemma 6.15]{Voi2}.

The image of the tangent space $T_X\NL(\gamma)$  at $[X]$ in $H^1(X,\Theta_X)$  equals 
\[ \{ \xi \in H^1(X,\Theta_X) \colon \xi\cup H^{k+1,k-1}\subset \gamma_{\prim}^\perp\}.\]
 From the construction of $I(\gamma)$ it follows now that the image of  $T_X\NL(\gamma)$ in $(S/J)_d$ equals $I(\gamma)_d/J_d$.
\end{proof}
\begin{definition}\label{defCom}
Let $U\subset U_d$ be a small neighborhood of $[X]$.
Pick  Hodge classes $\gamma_1,\dots,\gamma_r$ on $X$. Then the  \emph{Hodge locus of $\gamma_1,\dots,\gamma_r$} denoted by $\NL(\gamma_1,\dots,\gamma_r)$ is the locus in $U$ where all $\gamma _i$ remain a Hodge class. 
\end{definition}
\begin{remark}\label{rmkIns}
From the above discussion it is immediate that the codimension of $T_X\NL(\gamma_1,\dots,\gamma_r)$ in $T_X U_d$ equals the codimension of $\cap_{i=1}^r I(\gamma_i)_d$ in $S_d$. The latter equals the codimension of $\cap_{i=1}^r T_X\NL(\gamma_i)$ in $T_X U_d$.
\end{remark}

The following result should be well-known to experts, but we did not locate a precise reference in the literature:
\begin{proposition}\label{prpSmooth} Let $X$ be a hypersurface of degree at least $d\geq 2+\frac{2}{k}$ containing two $k$-planes $\Pi_1,\Pi_2$. Then the Hodge locus $\NL(\Pi_1,\Pi_2)$ is smooth in a neighborhood of $X$.
\end{proposition}
\begin{proof}
Let $I_1,I_2$ be the ideals associated with the Hodge classes $\Pi_1,\Pi_2$ respecitvely.
The tangent space of $\NL(\Pi_1,\Pi_2)$ has codimension $h_{I_1\cap I_2}(d)=h_{I_1}(d)+h_{I_2}(d)-h_{I_1+I_2}(d)$.

The ideals $I_1$ and $I_2$ are complete intersection ideals with $k+1$ generators of degree $1$ and $k+1$ generators of degree $d-1$. In this case the Koszul complex gives a resolution of $I_j$ and we find that
$h_{I_1}(d)=h_{I_2}(d)=\binom{k+d}{k}-(k+1)^2$.
The ideal $I_1+I_2$ is a complete intersection ideal with $k+c+1$ generators of degree $1$ and $k-c+1$ generators of degree $d-1$
Hence  the value of its Hilbert function at degree $d$  equals $\binom{k-c+d}{k-c}-(k+1-c)^2$. Hence the tangenst space of $\NL(\Pi_1,\Pi_2)$ has codimensiom
\[2 \binom{k+d}{k}-\binom{k-c+d}{k-c}-(k+1)^2-2c(k+1)+c^2\]

To calculate the dimension of $\NL(\Pi_1,\Pi_2)$ we first calculate the dimension of the space of $k$-planes in $\Ps^{2k+1}$ intersecting in dimension $k-c$. This is equivalent to choosing a $k-c+1$ dimensional subspace in $\C^{2k+2}$ and two additional $c$-dimensional spaces in the quotient space $\C^{k+c+1}$. Using the standard formula for the dimension of Grassmannians we obtain that this space has dimension
\[ (k+1-c)(k+1+c)+2c(k+1)=(k+1)^2-c^2+2c(k+1).\]

For a fixed pair of such planes $\Pi_1,\Pi_2$, the space of hypersurface containing $\Pi_1\cup \Pi_2$ has codimension
\begin{eqnarray*}h_{I(\Pi_1)\cap I(\Pi_2)}(d)&=&h_{I(\Pi_1)}(d)+h_{I(\Pi_2)}(d)-h_{I(\Pi_1)(d)+I(\Pi_2)}(d)\\&=&2\binom{k+d}{d}-\binom{k-c+d}{k-c}.\end{eqnarray*}
Since a hypersurface of degree at least 3 cannot contain a family of linear space of half its dimension, we find that the total codimension equals
\[ 2\binom{k+d}{d}-\binom{k-c+d}{k-c}(k+1)^2+c^2-2c(k+1).\]
In particular, the tangent space has the same dimension as the actual space and therefore is smooth.
\end{proof}
\begin{remark} 
Note that if $X$ contains two pairs of $k$-planes intersecting in codimension $c$ then the Zariski closure of $\NL(\Pi_1,\Pi_2)$ is singular at $X$, but $\NL(\Pi_1,\Pi_2)$ is smooth at $X$.
\end{remark}

\begin{definition}
Let  $\gamma_1,\dots,\gamma_r$ be Hodge classes on $X$.
Fix a point $(a_1:\dots :a_r)\in \Ps^{r-1}(\Q)$. 
We say that we have \emph{excess tangent dimension} at $(a_1:\dots:a_r)$ if
\[ T_X\NL(\gamma_1,\dots,\gamma_r) \subsetneq T_X\NL(a_1\gamma_1+\dots+a_r\gamma_r)\]
\end{definition}
The Hodge locus $\NL(\gamma)$ does not alter if we replace $\gamma$ by a multiple of $\gamma$, hence this locus depends only on the point in $\Ps^{r-1}(\Q)$.
In particular, if $\gamma\in H^{2k}(X,\Q)_{\prim}\cap H^{k,k}(X,\C)_{\prim}$ it makes sense to define $\NL(\gamma)$. 

We will now investigate whether for given $\gamma_1,\dots,\gamma_r$ there is \emph{generically} no excess tangent dimension, i.e., when there is a Zariski open  subset $U$ of $ \Ps^{r-1}(\Q)$ such that for all $a\in U$ there  is no excess tangent dimension.
In case we are merely interested in the existence of $U$ then we may restrict our analysis to the case $r=2$, 
since the tangent space of $\NL(\gamma_1,\dots,\gamma_r)$ is obtained by iterating intersections. I.e., we  aim to describe $T_X \NL(a_0\gamma_1+a_1 \gamma_2)$ for $(a_0:a_1)\in \Ps^1( \Q)$ in terms of $T_X \NL(\gamma_1)$ and $T_X\NL(\gamma_2)$. Since we are only interested when $a_0a_1\neq 0$ we may identify $(a_0:a_1)=(1:\lambda)$, with $\lambda\in \Q^*$, and restrict ourselves to loci of the shape $\NL(\gamma_1+\lambda \gamma_2)$.

\begin{lemma}\label{lemmakey}
Suppose $X\subset \Ps^{2k+1}$ is a smooth hypersurface of degree $d$ such that $X$ contains two subvarieties $Y_1,Y_2$ of dimension $k$. Suppose that $[Y_1]_{\prim}$ and $[Y_2]_{\prim}$ are linearly independent in $H^{2k}(X,\Q)_{\prim}$.

For $j=1,2$, let $I_j$ be the ideal associated with the Hodge class $[Y_j]$, cf. Construction~\ref{conIdeal}. Fix   isomorphisms $\sigma_j:(S/I_j)_{(k+1)(d-2)}\cong \C$. Let $\psi_j$ be the pairing
\[ (S/I_1\cap I_2)_d\times (S/I_1\cap I_2)_{kd-2k-2}\to  (S/I_1\cap I_2)_{(k+1)(d-2)} \to (S/I_j)_{(k+1)(d-2)}\stackrel{\sigma_j}{\longrightarrow}\C.\]

For $\lambda \in \Q^*$ let $I^{(\lambda)}$ be the ideal associated to $[Y_1]+\lambda [Y_2]$. Then there exists an injective function $\nu: \C^*\to\C^*$ such that
\[T_X\NL([Y_1]+\lambda[Y_2])/T_X\NL([Y_1],[Y_2])= I^{(\lambda)}_d/(I_1\cap I_2)_d=\ker_L (\psi_1+\nu(\lambda)\psi_2)\]
\end{lemma}

\begin{proof} By construction, the codimension of $I^{(\lambda)}_{(k+1)(d-2)}$ in  $S_{(k+1)(d-2)}$ equals one. 
From Lemma~\ref{lemHP} it follows that $(S/(I_1\cap I_2))_{(k+1)(d-2)}$ is two-dimensional.
Moreover, $I_1\cap I_2\subset I^{(\lambda)}$, hence
 the image of $I^{(\lambda)}_{(k+1)(d-2)}$ in   $(S/(I_1\cap I_2))_{(k+1)(d-2)}$ is one-dimensional. Let $g_\lambda$ be a generator for this one-dimensional space. 

Consider now the multiplication map
\begin{eqnarray*}
\mu: (S/I_1\cap I_2)_d \times (S/I_1\cap I_2)_{kd-2k-2}& \to &(S/I_1\cap I_2)_{(k+1)(d-2)}\\& \stackrel{\sim}{\longrightarrow} &(S/I_1)_{(k+1)(d-2)}\oplus (S/I_2)_{(k+1)(d-2)}\\
&\stackrel{(\sigma_1,\sigma_2)}{\longrightarrow} & \C\oplus \C\end{eqnarray*}
where the first map is the natural multiplication map, which is independent of $\lambda$ or of any choice,  the second map is a natural isomorphism and the third map is $(\sigma_1,\sigma_2)$, i.e., which depends on our choices. Let $p_j:\C\oplus \C\to\C$ be the projection on the $j$-th factor.
By construction $\psi_j$ equals $p_j \circ \mu$.

Recall that $I^{(\lambda)}_d/(I_1\cap I_2)_d$ equals the left kernel of the bilinear map
\[ (S/I_1\cap I_2)_d \times (S/I_1\cap I_2)_{kd-2k-2} \to ((S/I_1\cap I_2)_{(k+1)(d-2)})/g_{\lambda}\]
Let $(a,b) =(\sigma_1(g_\lambda),\sigma_2(g_\lambda))$. Then 
\[
 I^{\lambda}_d/(I_1\cap I_2)_d=\ker_L(b\psi_1-a\psi_2).\]

We would like to set $\nu(\lambda)=-a/b$. In order to do this we need to show that $b$ is nonzero. Moreover, in order to prove our claim we need to show that $\nu$ is injective, i.e., that we can recover $\lambda$ from $(a,b)$. We start by proving the last part:
Let $W$ be the kernel of $S_{(k+1)(d-2)}\to (S/I_1\cap I_2)_{(k+1)(d-2)} \to \C\oplus \C \to \C $, where the second map is just multiplication by $(b,-a)^T$. Then $W=I_{(k+1)(d-2)}^{(\lambda)}$. Let $W'\subset (S/J)_{(k+1)(d-2)}$ be its orthogonal complement under the multiplication map
\[ S_{(k+1)(d-2)}\times (S/J)_{(k+1)(d-2)}\to (S/J)_{(2k+2)(d-2)}\cong \C\]
Then $W'$ is generated by $[Y_1]+\lambda [Y_2]_{\prim}$. Now $[Y_1]+\lambda[Y_2]=\tau([Y_1]+\lambda'[Y_2])$ if and only if $\tau =1$ and $\lambda=\lambda'$.

In particular, for $\lambda\neq 0$, we have that $ab\neq 0$; the map $\nu:\lambda\mapsto -a/b$ is injective and
\[ I^{(\lambda)}_d/(I_1\cap I_2)_d=\ker_L (\psi_1+\nu(\lambda)\psi_2).\] 
The equality
\[  T_X\NL([Y_1]+\lambda[Y_2])/T_X\NL([Y_1],[Y_2])= I^{(\lambda)}_d/(I_1\cap I_2)_d\]
follows directly from the discussion before this lemma.
\end{proof}

\begin{theorem}\label{thmTsp}
Suppose $X\subset \Ps^{2k+1}$ is a smooth hypersurface of degree $d$ such that $X$ contains two subvarieties $Y_1,Y_2$ of dimension $k$. Suppose that $[Y_1]_{\prim}$ and $[Y_2]_{\prim}$ are linearly independent in $H^{2k}(X,\Q)_{\prim}$.

Let $I_j$ be the ideal associated with the Hodge class $[Y_j]$, cf. Construction~\ref{conIdeal}.
Suppose that the left kernel of the multiplication map
\[ (I_1+I_2/I_2)_d\times (I_1+I_2/I_2)_{kd-2k-2}\to (S/I_2)_{(k+1)(d-2)}\]
is zero. Then for all but finitely many $\lambda \in\Q^*$  we have
\[ \codim T_X\NL([Y_1]+\lambda [Y_2])=\codim T_X \NL([Y_1],[Y_2])\]

Moreover, if for both $j\in\{1,2\}$  the left kernels of 
\[ (I_1+I_2/I_j)_d\times (I_1+I_2/I_j)_{kd-2k-2}\to (S/I_j)_{(k+1)(d-2)}\]
are zero and $d=kd-2k-2$, then there are at most $h_{I_1+I_2}(d)$ values of $\lambda\in \Q^*$ such that
\[ \codim T_X\NL([Y_1]+\lambda [Y_2])<\codim T_X \NL([Y_1],[Y_2])\]
\end{theorem}
\begin{proof} Let $\lambda\in \Q^*$.
Let $I^{(\lambda)}=I([Y_1]+\lambda [Y_2])$.

From the previous lemma it follows that 
\[ T_X\NL([Y_1]+\lambda[Y_2])/T_X\NL([Y_1],[Y_2])\]
equals $\ker_L (\psi_1+\nu(\lambda)\psi_2)$.
The first statement now follows immediately.

From Lemma~\ref{lemSpectrumSize} it follows that for at most $\dim V-s_1-s_2$ nonzero values of $\lambda$ there is a rank drop. 

In our case $\dim \ker_L(\psi_j)=\dim V-r_j$. Since the left kernel is zero, the rank $s_i$ of $\varphi_i$  is maximal, hence $s_i=\dim V-r_j$. This implies that there are at most
\[ r_1+r_2-\dim V\]
exceptional values.
By construction $r_j=h_{I_j}(d)$ and $\dim V=h_{I_j\cap I_j}(d)$. Using Lemma~\ref{lemHP} we obtain
\[ r_1+r_2-\dim V=h_{I_1+I_2}(d).\]
\end{proof}

\begin{corollary} Using the notation of the previous theorem. Suppose, moreover, that $d=kd-2k-2$ and that $h_{I_1+I_2}(d)=0$ then for all $\lambda \in \Q^*$ we have 
\[ \codim T_X\NL([Y_1]+\lambda [Y_2])=\codim T_X \NL([Y_1],[Y_2]).\]
\end{corollary}

\begin{proof}
Note that from Lemma~\ref{lemKernelBnd} it follows that the left kernel of the multiplication map
\[ (I_1+I_2/I_2)_d\times (I_1+I_2/I_2)_{kd-2k-2}\to (S/I_2)_{(k+1)(d-2)}\]
has dimension at most
\[ h_{I_1}(d)+h_{I_2}(d)-h_{I_1\cap I_2}(kd-2k-2).\]
Using $d=kd-2k-2$ and  Lemma~\ref{lemHP} it follows that the above equals
\[ h_{I_1+I_2}(d),\]
which is zero by assumption. Hence we may apply the previous result to conclude that 
\[\dim T_X\NL(\gamma_1+\lambda\gamma_2) = \dim T_X\NL(\gamma_1,\gamma_2) \] 
holds for all $\lambda \in \Q^*$.
\end{proof}

We will now show a more general result for complete intersection cycles.
Recall that the Hodge locus associated with a $k$-dimensional complete intersection is smooth and irreducible, see e.g., \cite{KloCI} and that we described its associated ideal in Example~\ref{exaIdealCI}.

\begin{theorem}\label{ThmExcGen} Fix an integer $c$ such that $1\leq c \leq k+1$ be an integer, let  $d_0,\dots, d_k, e_0,\dots,e_{c-1}$ be integers.
Let $Z\subset \Ps^{2k+1}$ be a complete intersection of multidegree $(d_0,\dots,d_k,e_0,\dots,e_{c-1})$, let $Y_1$ and $Y_2$ be complete intersections of multidegree $(d_0,\dots, d_k)$ and $(e_0,\dots,e_{c-1},d_c,\dots,d_k)$ respectively, such that their intersection equals $Z$. 
 If  
 \[ \sum_{i=0}^{c-1} (e_i+d_i)<(c-1)d\]
 then there is no excess tangent dimension for $[Y_1]+\lambda [Y_2]$ for all but finitely many $\lambda \in \Q^*$.
\end{theorem}

\begin{proof}
 Let $I_j$ be ideal associated with the Hodge class $[Y_j]$. From Example~\ref{exaIdealCI} and the results in Subsection~\ref{subsecCII} it follows that $I_1+I_2$ is a complete intersection ideal with generators of degree $(e_0,\dots,e_{c-1},d_0,\dots,d_{c-1},d_c,\dots,d_k,d-d_c,\dots,d-d_k)$.
Hence the socle degree  of $S/I_1+I_2$ equals
\[(k-c+1)d-2k-2+ \sum_{i=0} ^{c-1}(e_i+d_i).\]
This is less than $kd-2k-2$ if and only if
 \[ \sum_{i=0}^{c-1} (e_i+d_i)<(c-1)d.\]
 Hence from Lemma~\ref{lemHighDeg} it follows that the multiplication map from Theorem~\ref{thmTsp} has no left kernel. From the same Theorem the result now follows.
\end{proof}

\begin{example}[Independence of $X$ for small $k$]\label{exaRestriction} 
Suppose now that both $\gamma_i$ are  classes of  $k$-planes, and that the intersection of these planes has dimension $k-c$.
In this case we  have that $I_1+I_2$ is generated by $k+1+c$ linear forms and $k+1-c$ forms of degree $(d-1)$. In particular $h_{I_1+I_2}(t)=0$ for $t>(k+1-c)(d-2)$.

Hence if
\[ kd-2k-2>(k+1-c)(d-2)\]
then there is no left kernel by Remark~\ref{remb01}. This is equivalent with
$(c-1)(d-2)>2$.

Since we assumed that $d\geq 2+\frac{2}{k}$ holds, we can only have excess tangent dimension for infinitely many choices of $\lambda$ if one of $c=1$; $(c-1)(d-2)=1$ or $(c-1)(d-2)=2$ holds.
In particular, if $(c-1)(d-2)=2$ then $h_{I_1+I_2}(kd-2k-2)=1$.

We specialize even further:
If we assume moreover that $d>kd-2k-2$ holds then $d>(k+c-1)(d-2)$. In particular, we have that
\[ \dim (S/I_1\cap I_2)_d>\dim (S/I_1\cap I_2)_{kd-2k-2}.\]
Therefore the pairing
\[  (S/I_1\cap I_2)_d\times (S/I_1\cap I_2)_{kd-2k-2} \to (S/I_1\cap I_2)_{(k+1)(d-2)}\to \C\]
has a left kernel for all $\lambda$ and therefore 
\[ T_X\NL([\Pi_1],[\Pi_2])\subsetneq T_X\NL([\Pi_1]+\lambda[\Pi_2])\]
for all $X\in \NL([\Pi_1],[\Pi_2])$.

This happens precisely when $k\leq \frac{d+1}{d-2}$, i.e., $d=3, k\in \{2,3,4\}$ and $d=4, k\in\{1,2\}$.
Hence for $d=3, k=2,3,4$ we have excess tangent dimension for all $\lambda$,  as for $d=4,k=1,2$ and all $\lambda$, and independently of $X$.

This is a generalization of Movasati's result for the case where $X$ is the Fermat hypersurface. \cite{MovConb}
\end{example}

\section{Movasati's conjecture}\label{secMov} 
In this section we will give  counterexamples to Movasati's conjecture. Movasati's conjecture is formulated for $d=3,c=3$ and $k\geq 3$. However, our counterexamples require  $k\geq 5$. More precisely,  we will show that for $d=3,c=3,k=5$ there is a cubic tenfold in $\Ps^{11}$ containing two $5$-planes intersecting in dimension 2, such that $T_X\NL([\Pi_1]+\lambda [\Pi_2])=T_X\NL([\Pi_1],[\Pi_2])$ holds for all but one $\lambda \in \Q^*$.
This implies that for $k=5$ Movasati's conjecture is false. We will then apply a dimension raise trick to give a counterexample for all $k\geq 6$. Along the same lines we will construct a counterexample for $d=4,c=2$ and $k\geq 3$ and for $d=3, c=2$ and $k\geq 5$.

Moreover, in  the next section we will show that the Hodge locus is reducible in a neighbourhood of the Fermat point, also explaining why there is a difference between the actual dimension and the dimension of the tangent space at the Fermat point, the latter calculated by Villaflor in \cite[Section 5.10]{MovConb}.

Let us consider a general hypersurface $X$ of degree $d$ in $\Ps^{2k+1}$ containing two $k$-planes intersecting in codimension $c$. Then after a coordinate change we may assume that $X=V(f)$ with
\[ f=\sum_{i=0}^ {c-1}\sum_{j=c}^{2c-1} x_ix_j Q_{ij}+\sum_{i=k+c+1}^{2k+1} x_i P_{i}.\]
In this case 
\[ I_1+I_2=\langle x_0,\dots,x_{2c-1},x_{k+c+1},\dots,x_{2k+1},P_{k+c+1},\dots,P_{2k+1}\rangle\]
and
\[I_2=\langle x_0,\dots,x_{c-1},(\sum_{j=c}^{2c-1} Q_{ij}x_j)_{i=0}^{c-1},x_{k+c+1},\dots,x_{2k+1},P_{k+c+1},\dots,P_{2k+1}\rangle\]

We will now calculate the left kernel of the pairing 
\[ ((I_1+I_2)/I_2)_d\times ((I_1+I_2)/I_2)_{kd-2k-2} \to (S/I_2)_{(d-2)(k+1)}\]
for  various explicit choices of $Q_{ij}$ and $P_i$.

\subsection{Fermat case}\label{ssFermat}
The first example is the Fermat hypersurface $X$ in $\Ps^{2k+1}$, i.e., $X=V(f)$ with
\[f= \sum_{i=0}^{2k+1}  x_i^d.\]

We follow the examples of \cite[Chapter 16]{MovBook}. Let
$\zeta$ a primitive $2d$-th root of unity. Let $\Pi_1$ be the $k$-plane  
\[ V(\{x_{2i}-\zeta x_{2i+1}\colon i=0,\dots,k\}).\]
Fix an integer $c$ such that $1\leq c\leq k+1$. For $i=0,\dots, c-1$ fix an odd integer $\alpha_i$ such that $3\leq \alpha_i\leq 2d-1$. Let $\alpha_i=1$ for $i=c,\dots,k$. Let $\Pi_2$ be the $k$-plane
\[ V(\{x_{2i}-\zeta^{\alpha_i} x_{2i+1}\colon i=0,\dots,k\}).\]
In the way $\Pi_1,\Pi_2\subset X$ and $\dim \Pi_1\cap \Pi_2=k-c$.

\begin{proposition} \label{prpFermat} Let $X,\Pi_1,\Pi_2$ as above. For $j=1,2$ let $I_j$ be the ideal associated with the Hodge class $\Pi_j$.
The dimension of the left kernel of the multiplication map
\[ ((I_1+I_2)/I_2)_d\times ((I_1+I_2)/I_2)_{kd-2k-2} \to (S/I_2)_{(d-2)(k+1)}\]
equals one if $(c-1)(d-2)=2$ and equals $k+1-c$ if $(c-1)(d-2)=1$.
\end{proposition}
\begin{proof}
We take new coordinates such that
\begin{eqnarray*}
 y_{2i}=x_{2i}-\zeta x_{2i+1}&&\mbox{for }i=0,\dots,k+1;\\
 y_{2i+1}=x_{2i}-\zeta^{\alpha_i} x_{2i+1}&&\mbox{for }i=0,\dots,k-c;\\
 y_{2i+1}=x_{2i+1} &&\mbox{for }i=k-c+1,\dots,k+1.\end{eqnarray*}

In this new coordinates we have
\[ f= \sum_{i=0}^{k-c} y_{2i}y_{2i+1}g_i(y_{2i},y_{2i+1})+\sum_{i=k-c+1}^k y_{2i} h_i(y_{2i},y_{2i+1}).\]
Hence $I_1$ is generated by $y_{2i}$ and $y_{2i+1}g_i$ for $i=0,\dots,k-c$ and $ y_{2i}, h_i$ for $i=k-c+1,\dots,k+1$.
However, \[y_{2i+1}g_i(y_{2i},y_{2i+1})\equiv \beta y_{2i+1}^{d-1} \bmod y_{2i} \mbox{ for }i=0,\dots,k-c,\] for some nonzero $\beta \in \Q$.
A similar congruence holds for larger $i$. In particular,
\[ I_1=\langle y_{2i},y_{2i+1}^{d-1}\colon i=0,\dots, k\rangle.\]
Similarly, we find that
\[ I_2=\langle y_{2i+1},y_{2i}^{d-1}\colon i=0,\dots, k-c\rangle \cap \langle y_{2i},y_{2i+1}^{d-1}\colon i= k-c+1,\dots ,k\rangle.\]
In particular,
\[ I_1+I_2=\langle y_0,\dots, y_{2k-2c+1}\rangle \cap \langle y_{2i},y_{2i+1}^{d-1} \colon i=k-c+1,\dots, k\rangle .\]
After a permutation of coordinates we may assume that 
\[ I_1+I_2=\langle y_0,\dots,y_{2c-1}, y_{2c}^{d-1},\dots,y_{k+c}^{d-1},y_{k+c+1},\dots,y_{2k+1}\rangle\]
and
\[ I_2=\langle y_0^{d-1},\dots,y_{c-1}^{d-1},y_c,\dots,y_{2c-1},y_{2c}^{d-1},\dots,y_{k+c}^{d-1},\dots,y_{2k+1}^{d-1}\rangle \]
Since both ideals are monomial ideals, there is a monomial basis for $I_2+I_1/I_2$. This basis  consists of monomials $\prod_{i=0}^{2k+1} y_i^{a_i}$
such that 
\begin{enumerate}
\item there is at least one $0\leq i<c$ with $a_i>0$;
\item $a_i=0$ for $c\leq i<2c$ and $k+c+1\leq i \leq 2k+1$;
\item $a_i<d-1$ for $0\leq i \leq c-1$, $2c\leq i \leq k+c$.
\end{enumerate}

The topdegree $(S/I_2)_{(k+1)(d-2)}$ is generated by the single monomial
\[M=(y_0\dots y_{c-1} y_{2c}\dots y_{k+c})^{d-2}.\]
Let $N$ be a monomial contained in $(I_1+I_2)_d$ but not in $(I_2)_d$.
Let  $N'=M/N$. Then $N'$ is part of the monomial basis for $(S/I_2)_{kd-2k-2}$. Moreover, for any nonzero monomial $N''$ in $(S/I_2)_{kd-2k-2}$ we have that $N N''=0$ in $(S/I_2)_{(k+1)(d-2)}$ unless $N''\equiv N' \bmod I_2$.

From this it follows immediately that $N\in \ker_L$ if and only if $N'\not \in I_2$.
Consider now
\[N=(y_0\dots y_{c-1})^{d-2}\]
then this is a monomial of degree $c(d-2)$. 

Suppose first that $(c-1)(d-2)=2$ holds. Then $c(d-2)=d$. 
In this case we have $N'=\prod_{i=2c}^{k+c} y_i^{d-2}$ which is not $I_2$. 
In particular the left kernel has  dimension at least one.

If $(c-1)(d-2)=1$ then $c(d-2)=d-1$. In this case  we have that for $j=2c,\dots,k+c$ we have $y_jN$ is a monomial of degree $d$ and is part of our basis. For each such monomial we have that that $M/N' \not \in I_2$.
Hence the left kernel has dimension at least $k-c+1$.

The rank of the multiplication maps of $S/I_j$ equals $h_{I_j}(d)=h_{I_j}(kd-2k-2)$. From Lemma~\ref{lemKernelBnd} it follows that the left kernel has dimension at most
\[ h_{I_1\cap I_2}(kd-k-2)-h_{I_1}(kd-k-2)-h_{I_2}(kd-k-2),\]
which  equals $h_{I_1+I_2}(kd-k-2)$. If $(c-1)(d-2)=2$ the latter quantity is one, if $(c-1)(d-2)=1$ the latter quantity is $k+1-c$. 
\end{proof}

\begin{example}
The above lemma shows only that if  $(c-1)(d-2)\leq 2$ then we cannot apply Lemma~\ref{lemgenrk}. In order to conclude that $\dim T_X \NL([\Pi_1]+\lambda [\Pi_2])>\dim T_X\NL([\Pi_1],[\Pi_2])$ holds, i.e., to reprove \cite[Proposition 5.6]{MovConb},  we need to show that the  left kernel of  $\mu: (S/I_1\cap I_2)_d\times(S/I_1\cap I_2)_{kd-2k-2} \to (S/I_1\cap I_2)_{(k+1)(d-2)}\stackrel{\sigma}{\to} \C$ is nonzero, under variation of the map $\sigma:(S/I_1\cap I_2)_{(k+1)(d-2)}\to \C$.
However, since $I_1\cap I_2$ is a monomial ideal this is straightforward: 

We continue to use the notation of the previous proof.
Let $M_1=(y_0\dots y_{c-1})^{d-2}$. 
Let $M_2=(y_c\dots y_{2c-1})^{d-2}$ and let $M=(y_{2c}\dots y_{k+c})^{d-2}$.
Then $\deg M M_1=\deg MM_2=(k+1)(d-2)$. Moreover, $MM_1$ and $MM_2$  form a basis for $(S/I_1\cap I_2)_{(d-2)(k+1)}$.

Fix now a quotient map $\sigma: (S/I_1\cap I_2)_{(k+1)(d-2)}\to \C$, such that $MM_1$ and $MM_2$ are not in the kernel of $\sigma$.
Let $\lambda \in \C^*$ such that $MM_2-\lambda MM_1$ is a generator of the kernel of $\sigma$.

 Let $N$ be a monomial of degree $2-(c-1)(d-2) \in \{0,1,2\}$, such that $N$ divides $M$.
 Then $\deg NM_1=\deg NM_2=d$. 
 
Using that  $I_1\cap I_2$ is a monomial ideal we find that if $N'\in  (S/I_1\cap I_2)_{kd-2k-2}$ is a monomial such that $N'(\lambda NM_2-NM_1)$ is nonzero in $S/(I_1\cap I_2)$ then $NN'$ divides $M$. In particular $N'=M/N$. However, then $\sigma (N'(\lambda NM_2-NM_1))=0$.

In particular  $\lambda NM_2-NM_1$ is in the left kernel of $\mu$. We need to show that this element is nonzero in $(S/I_1\cap I_2)_d$. Note that for $\{i,j\}=\{1,2\}$ we have that  $NM_i$ is nonzero in $S/I_i$ but is zero in $S/I_j$. Hence $\lambda NM_2-NM_1$ is nonzero in $S/(I_1\cap I_2)_d$. In particular, the left kernel is nonzero.
 \end{example}

\subsection{Counterexamples in low dimension}
The following three examples are the main building blocks for the counterexamples to Movasati's conjecture.
\begin{example}\label{ExaConA}
Take  $d=4,c=2,k=3$. 
Let $f$ be in the usual form
\[ \sum_{i=0}^1\sum_{j=2}^3 x_i x_j Q_{i,j} +\sum_{i=6}^7 x_i P_i.\]

Pick now the following $Q_{ij}\in \C[x_0,\dots,x_7]_2$, $P_i \in \C[x_0,\dots,x_7]_3$ such that  
\[
 Q_{02} =x_0^2+x_2^2+x_4^2;
 Q_{13} =x_1^2+x_3^2+x_5^2;
 P_6 =x_4(x_4^2+x_6^2);
 P_7 = x_5(x_5^2+x_7^2);
\]
and $Q_{12} =Q_{03}= 0$.
Then $X=V(f)$ is smooth.  Note that if we would replace $Q_{02}$ by $x_0^2+x_2^2$ and $Q_{13}$ by $x_1^2+x_3^2$ then we would obtain an hypersurface isomorphic to the Fermat cubic, hence we are considering a two-monomial deformation of a hypersurface isomorphic with the Fermat hypersurface. Moreover, $X$ admits an automorphism mapping $(x_0,x_1,x_2,x_3)\mapsto (x_2,x_3,x_0,x_1)$ and leaving the other variables invariant. This automorphism swaps both $3$-planes.
 
 Since $d=4,k=3$ we have   the equality $d=kd-2k-2=4$. In order to apply Theorem~\ref{thmTsp} we need to consider the symmetric bilinear form on $(S/I_1)_4$ given by multiplication. More  concretely, we want to calculate the left kernel of the  multiplication map $((I_1+I_2)/I_1)_4\times ((I_1+I_2)/I_1)_4 \to (S/I_1)_8$.
 
 The image of the ideal $I_1+I_2$ in $S/(x_0,x_1,x_6,x_7)$ is generated by $x_2,x_3,x_4^3,x_5^3$.
  Hence a $\C$-basis for $(S/I_1+I_2)_4$ is given by $\prod_i x_i^{a_i}$ such that $\sum a_i=4$ and one of $a_2>0,a_3>0,a_4>2,a_5>2$ holds.
  Hence the vector space $((I_1+I_2)/I_1)_4$ is generated by $\prod_{i=2}^5 x_i^{a_i}$ such that $a_2>0$ or $a_3>0$. Moreover, $a_4<3,a_5<3$ should hold.
 Using the congruences $x_2^3\equiv-x_2x_4^2$ and $x_3^3\equiv -x_3x_5^3$ we may assume that 
 $a_2<3$ and $a_3<3$. In this way we find  18 monomials which can be easily checked  to be linearly independent, i.e., their residue classes form a basis for $((I_1+I_2)/I_1)_4$.
    Similarly, one can show that the vector space $(S/I_2)_8$ is spanned by $M=(x_2x_3x_4x_5)^2$.
  
  The Gram matrix of this paring is a $18\times 18$ matrix, and is too large to include in this paper. One can determine  easily the nonzero entries by hand or by computer.
  For most monomials $N$ in our basis we have that $M$ and $M/N$  pair to $M$; that $M/N$ is also in our basis and $M/N$ is different from $N$. Moreover, for any other monomial $N'$ in our basis we have $NN'=0$. In this situation the pairing on $\{N,M/N\}$ has matrix
  \[ \begin{pmatrix} 0&1\\1&0\end{pmatrix}\]
  and the subspace generated by $N,M/N$ is orthogonal to the space generated by the 16 other monomial.
  
  We now list the monomials which are exceptions to this situation. The first exception is $N=x_2x_3x_4x_5$. In this case we have $N=M/N$ and that $N$ is orthogonal to all other monomials in our basis.
  
  The second exception is  $N=x_2^2x_3^2$. In this case $M/N\not \in I_2$. However the pairing of $N$ with itself is nonzero, as is the pairing with $x_3^2x_4^2$ and $x_2^2x_5^2$. 
   These three monomials generate a dimension 3 subspace of $(I_1+I_2/I_1)_4$ orthogonal to the 15 other monomials. One easily checks that  the Gram matrix on these monomials  has rank 3.
   
   The final two exceptions are the space generated by $\{N,M/N\}$ where $N$ is either  
  $x_2x_3^2x_4$ or
  $x_2^2x_3x_5$. These monomials have a nonzero pairing with themselves but are $M/N$ has zero self pairing, moreover $N,M/N$ are orthogonal to the other $16$ monomials. Here we  get for both $N$ that the pairing on $N,M/N$ is
\[\begin{pmatrix}
a& 1\\1& 0\end{pmatrix}.\]
In particular, the Gram matrix has full rank 18 and therefore there is no left kernel. For the computer code used see Appendix~\ref{secCode}.
\end{example}

\begin{example}\label{ExaConB}
We want also to present an example for $c=d=3$. The smallest possible value for $k=5$.

For $i=0,1,2$ let $Q_{i,i+3}=x_i+x_{i+3}+x_{i+6}$. For all $(i,j)$ such that $i\in \{0,1,2\}$, $j\in \{3,4,5\}$ let $Q_{i,j}=0$. For $i=9,10,11$ let $P_i=x_{i-3}(x_{i-3}+x_i)$.
\[ f=\sum_{i=0}^2 x_ix_{i+3}Q_{i,i+3}+\sum_{i=9}^{11} x_iP_i.\]
Then $X=V(f)$ is smooth. Moreover $X$ admits an automorphism mapping $(x_0,\dots,x_{11})\mapsto (x_3,x_4,x_5,x_0,x_1,x_2,x_6,\dots,x_{11})$, swapping both planes.

Similarly as above  one can determine the Gram matrix of the product map, and calculate its rank, which turns out to be maximal. For the computer code used see Appendix~\ref{secCode}.
\end{example}

\begin{example}\label{ExaConC} Our final example is an example with $c=2$,$d=3$ and $k=5$.
Here we take for $i=0,1$
\[ Q_{i,i+2}=x_i+ x_{i+2}+x_{i+4}+x_{i+6};\]
$Q_{0,3}\equiv 3x_3+x_5$ and $Q_{1,2}\equiv x_6+5x_7$.
For $i=8,9,10,11$ take
\[P_i=x_{i-4}(x_{i-4}+x_i)\]
and let
\[ f=\sum_{i=0}^1\sum_{j=2}^3 x_ix_jQ_{i,j}+\sum_{i=8}^{11} x_iP_i.\]
Then $X=V(f)$ is smooth and admits an automorphism swapping both planes. (The coefficient in front of $x_3$ in $Q_{0,3}$ and $x_7$ in $Q_{1,2}$ are needed to obtain a smooth hypersurface.)
Similarly as above,  one can determine the Gram matrix of the product map, and calculate its rank, which turns out to be maximal.
For the computer code used see Appendix~\ref{secCode}.\end{example}

\begin{proposition}\label{prpIndBase} Suppose $d=3, c\in \{2,3\},k=5$ or $d=4,c=2,k=3$. Then there exists a hypersurface $X\subset\Ps^{2k+1}$ of degree $d$ containing two $k$-planes intersecting in codimension $c$. Such that for almost all $\lambda\in\Q$ we have
\[ \codim T_X\NL([\Pi_1]+\lambda [\Pi_2])=\codim T_X \NL([\Pi_1],[\Pi_2])\]
In particular, $\NL([\Pi_1]+\lambda[\Pi_2])=\NL([\Pi_1],[\Pi_2])$ in a neighborhood of $X$.

Moreover, if $(c-1)(d-2)=2$ then there is at most one $\lambda\in \Q^*$ such that 
\[ \codim T_X\NL([\Pi_1]+\lambda [\Pi_2])\neq\codim T_X \NL([\Pi_1],[\Pi_2]).\]
\end{proposition}
\begin{proof}
Consider the previous three examples. For each $X$ we have that the left kernel of the pairing form Theorem~\ref{thmTsp} is zero. Hence by Theorem~\ref{thmTsp} we find
\[ \codim T_X\NL([\Pi_1]+\lambda [\Pi_2])=\codim T_X \NL([\Pi_1],[\Pi_2]).\]
for almost all $\lambda\in \Q^*$.
Since $\NL([\Pi_1],[\Pi_2])$ is smooth (Proposition~\ref{prpSmooth}) and we have the obvious inclusion $\NL([\Pi_1],[\Pi_2])\subset \NL([\Pi_1]+\lambda [\Pi_2])$ we obtain the two inequalities
\begin{eqnarray*} 
\dim \NL([\Pi_1],[\Pi_2])&\leq& \dim \NL([\Pi_1]+\lambda [\Pi_2]) \mbox{ and }\\
\dim \NL([\Pi_1],[\Pi_2]) &=&\dim T_X\NL([\Pi_1],[\Pi_2]) =\dim T_X\NL([\Pi_1]+\lambda [\Pi_2]) \\&\geq& \dim \NL([\Pi_1]+\lambda [\Pi_2]).\end{eqnarray*}
In particular, $\NL([\Pi_1],[\Pi_2])=\NL([\Pi_1]+\lambda [\Pi_2])$ for all but finitely many $\lambda \in \Q$.

Finally, note that all three examples admit an automorphism swapping both planes. This implies that also the left kernel of the second pairing is zero. From Theorem~\ref{thmTsp} it follows that for at most
$ h_{I_1+I_2}(d)$
nonzero values of $\lambda$ there is a strict inequality. Using $(c-1)(d-2)=2$  we find that $h_{I_1+I_2}(d)=1$. 
\end{proof}

The above proposition yields counterexamples to Movasati's conjecture for $k=5$ (when $d=3$, $c=2,3$) and $k=2$ (when $d=4,c=2$).
To provide counterexample in higher dimension we apply the following construction:
Let $X=V(f)$ be a smooth hypersurface of degree $d$ in $\Ps^{2k+1}$ containing two $k$-dimensional subvarieties $Y_1,Y_2$ with associated classes $\gamma_1,\gamma_2$ in cohomology. 
Consider $\tilde{f}=f+x_{2k+3}^d+x_{2k+3}x_{2k+2}^{d-1}$. 
And consider $\tilde{X}=V(\tilde{f})\subset\Ps^{2k+3}$. Then the cone $C(Y_j)$ over  $Y_j$ in $V(x_{2k+3})\cong \Ps^{2k+2}$ is contained in $\tilde{X}\cap V(x_{2k+3})$ and therefore in $\tilde{X}$. In this way we obtained two $k+1$ dimensional subvarieties $\tilde{Y}_1,\tilde{Y}_2$ of $\tilde{X}$, with associated classes $\tilde{\gamma_i}$.
Let $\tilde{S}=\C[x_0,\dots,x_{2k+3}]$ and let $S=\C[x_0,\dots,x_{2k+1}]$ as before.

\begin{proposition} \label{prpDimTrick}
Let $0\leq \delta\leq d$ be an integer. 
The dimension of the left kernel of the multiplication map
\[  (I(\tilde{\gamma}_1)+I(\tilde{\gamma}_2)/I(\tilde{\gamma}_2))_{d-\delta} \times (I(\tilde{\gamma}_1)+I(\tilde{\gamma}_2)/I(\tilde{\gamma}_2))_{(k+1)d-2k-4+\delta} \to (\tilde{S}/I(\tilde{\gamma}_2))_{(k+2)(d-2)} \]
equals the sum of the dimensions of the kernel of 
\[ ( I(\gamma_1)+I(\gamma_2)/I(\gamma_2))_{d-\delta'} \times (I(\gamma_1)+I(\gamma_2)/I(\gamma_2))_{kd-2k-2+\delta'} \to (S/I(\gamma_2))_{(k+1)(d-2)},\]
where $\delta'$ runs over the integers such that $\delta\leq \delta'\leq 2-(c-1)(d-2)$.
\end{proposition}

\begin{proof}
Recall that for $j=1,2$ the cycle $\tilde{\gamma_j}$ is a complete intersection cycle. From Example~\ref{exaIdealCI} it follows immediately that
\[\tilde{I_i}:= I(\tilde{\gamma_i})=\tilde{S} I(\gamma_i)+\langle x_{2k+2}^{d-1},x_{2k+3}\rangle.\]
In particular,
\[ (\tilde{S}/\tilde{I_j})_r =\oplus_{i=0}^{d-2} x_{2k+2}^i (S/I(\gamma_j))_{r-t}\]
and similar decompositions holds for $\tilde{S}/(\tilde{I_1}\cap \tilde{I_2})$ and $\tilde{S}/(\tilde{I_1}+\tilde{I_2})$.

The multiplication map
\[ (\tilde{S}/\tilde{I}_2)_{d-\delta} \times (\tilde{S}/\tilde{I}_2)_{(k+1)d-2(k+1)-2+\delta}\to (\tilde{S}/\tilde{I}_2)_{(k+2)(d-2)}\]
can be decomposed as follows: Let $R=S/I_2$. We can write the previous map as the sum over $\alpha,\beta$ of 
\[ x_{2k+2}^\alpha R_{d-\delta-\alpha} \times x_{2k+2}^\beta R_{(k+1)(d-2)-2-\beta+\delta}\to x_{2k+2}^{\alpha+\beta} R_{(k+2)(d-2)-\alpha-\beta}.\]
The right hand side is zero if $\alpha+\beta\neq d-2$: If $\alpha+\beta>d-2$ then $x_{2k+2}^{\alpha+\beta} =0$ in $\tilde{S}/\tilde{I_2}$, if $\alpha+\beta<d-2$ then $R_{(k+2)(d-2)-\alpha-\beta}=0$, since $R$ is the Artinian Gorenstein algebra associated to a Hodge class (cf. Remark~\ref{rmkSocDeg}).
Hence the multiplication map simplifies to 
\[\oplus_{\alpha\geq 0} x_{2k+2}^\alpha R_{d-\delta-\alpha} \times x_{2k+2}^{d-2-\alpha} R_{kd-2(k+1)+\alpha+\delta}\to x_{2k+2}^{d-2} R_{(k+1)(d-2)}\]
For the rest of the prove we take $\alpha=\delta'-\delta$. Then $\delta'\geq \delta$ since $\alpha\geq 0$.

We want to determine the left kernel when the pairing is restricted to the image of 
\[x_{2k+2}^{\delta'-\delta} (I(\gamma_1))_{d-\delta'}\times x_{2k+2}^{d-2-\delta'+\delta} (I(\gamma_1))_{kd-2(k+1)+\delta'}.\]
By Lemma~\ref{lemHighDeg} there is no left kernel if
\[ h_{I(\gamma_1)+I(\gamma_2)}(kd-2(k+1)+\delta')=0.\]
By Example~\ref{exaCI} this happens if and only if
\[ kd-2(k+1)+\delta'>(k+1-c)(d-2).\]
This latter inequality simplifies to $\delta'>2-(c-1)(d-2)$.
The result is now immediate.
\end{proof}

\begin{remark} For most choices of $c,d,\delta$ we sum over the empty set, since $\delta>2-(c-1)(c-2)$ and therefore there is no left kernel.
\end{remark}

\begin{theorem}\label{mainThm} Suppose $d=3, c\in \{2,3\} ,k\geq 5$ or $d=4,c=2,k\geq 3$. Then there exists a hypersurface $X\subset\Ps^{2k+1}$ of degree $d$ containing two $k$-planes intersecting in codimension $c$, such that for almost all $\lambda\in\Q$ we have
\[ \codim T_X\NL([\Pi_1]+\lambda [\Pi_2])=\codim T_X \NL([\Pi_1],[\Pi_2])\]
In particular, $\NL([\Pi_1]+\lambda[\Pi_2])=\NL([\Pi_1],[\Pi_2])$ in a neighborhood of $X$.

Moreover, if $(c-1)(d-2)=2$ then there is at most one nonzero $\lambda\in \Q$ such that 
\[\NL([\Pi_1]+\lambda[\Pi_2])\neq \NL([\Pi_1],[\Pi_2])\] in a neighborhood of $X$.
\end{theorem}
\begin{proof}
We start with the first statement. 
We are going to prove by induction on $k$ that there exists an example such that
 the multiplication map
\[  (I({\gamma}_1)+I({\gamma}_2)/I({\gamma}_2))_{d-\delta} \times (I({\gamma}_1)+I({\gamma}_2)/I({\gamma}_2))_{kd-2k-2+\delta} \to ({S}/I({\gamma}_2))_{(k+1)(d-2)} \]
has no left kernel for all $\delta$.

Note that if the hypothesis holds for $\delta=0$ and fixed $k$ then it holds for fixed $k$ and all $\delta$ such $0\leq \delta \leq d$ by Lemma~\ref{lemDegShift}. Hence we may restrict to $\delta=0$.
%
The base case for the induction follows from Examples~\ref{ExaConA} through~\ref{ExaConC}, whereas the induction step is Proposition~\ref{prpDimTrick}.
Using Theorem~\ref{thmTsp} and Proposition~\ref{prpSmooth} we obtain the following inequalities for almost all $\lambda\in \Q^*$:
\begin{eqnarray*} 
\dim \NL([\Pi_1],[\Pi_2])&\leq& \dim \NL([\Pi_1]+\lambda [\Pi_2]) \mbox{ and }\\
\dim \NL([\Pi_1],[\Pi_2]) &=&\dim T_X\NL([\Pi_1],[\Pi_2]) =\dim T_X\NL([\Pi_1]+\lambda [\Pi_2]) \\&\geq& \dim \NL([\Pi_1]+\lambda [\Pi_2]).\end{eqnarray*}
In particular, $\NL([\Pi_1],[\Pi_2])=\NL([\Pi_1]+\lambda [\Pi_2])$ for all but finitely many $\lambda \in \Q$.

To prove the Moreover part, note that we make here the extra assumption  $(c-1)(d-2)=2$. We will proceed by induction, and the base case follows again from Examples~\ref{ExaConA} through~\ref{ExaConC} together with  Proposition~\ref{prpIndBase}. For the induction step, recall that we are considering a family of multiplication maps
\[ \varphi_\lambda: S/(I_1\cap I_2)_d\times S/(I_1\cap I_2)_{(k+1)d-2k-4} \to S/(I_1\cap I_2)_{(k-1)(d-1)} \to \C\]
where the left hand side  can be decomposed as a direct sum 
\[ S/(I_1\cap I_2)_{d-\alpha} \times S/(I_1\cap I_2)_{kd-2k-2+\alpha} \to S/(I_1\cap I_2)_{(k-2)(d-1)} \to \C\]
If $\alpha$ is nonzero then this summand has full rank for every nonzero value of $\lambda$ by Lemma~\ref{lemHighDeg}. Hence only the summand for which $\alpha=0$ can contribute to a rank drop.
In particular the number of $\lambda$, where the rank drops, does not increase. 
\end{proof}

\section{Split case}\label{secSplit}
In this section we will give a geometric explanation for the increase of the dimension of the tangent space at the Fermat point. More precisely, we will show that along a small codimension locus of $\NL([\Pi_1],[\Pi_2])$ the locus $\NL([\Pi_1]+\lambda [\Pi_2])$ is reducible for almost all $\lambda$, and that $\NL([\Pi_1],[\Pi_2])$ is the largest dimensional component.

\begin{definition}
Let $X\subset \Ps^{2k+1}$ be a hypersurface containing two $k$-planes intersecting in an $m$-plane. Let $c=k-m$.

We say that $X$ is \emph{split of codimension $c$} if there exist a coordinate transformation of $\Ps^{2k+1}$ such that $X$ is the zero set of $f\in \C[x_0,\dots,x_{2k+1}]$ of the form
\[ f=\sum_{i=0}^{c-1}\sum_{j=c}^{2c-1} x_i x_j Q_{ij}+\sum_{j=k+2+c}^{2k+1} x_jP_j\]
and $Q_{ij}\in \C[x_0,\dots,x_{2c-1}]$ for $0\leq i \leq c-1, c\leq j \leq 2c-1$.
\end{definition}

\begin{remark} If $X$ contains  two $k$-planes intersecting in an $(k-c)$-plane then we can always find a change of coordinates such that $X=V(f)$, with 
\[ f=\sum_{i=0}^{c-1}\sum_{j=c}^{2c-1} x_i x_j Q_{ij}+\sum_{j=k+2+c}^{2k+1} x_jP_j\]
and $Q_{ij}\in \C[x_0,\dots,x_{k+c+1}]$. In order to be of split type requires the $Q_{ij}$ to be contained in the subring  $\C[x_0,\dots,x_{2c-1}]$.

However, our notion of split type is weaker than the notion needed to apply the Thom-Sebastiani theorem, where one additionally requires  that $P_j$ is in the smaller ring $ \C[x_{2c},\dots,x_{2k+1}]$.

From the description of the Fermat hypersurface in Subsection~\ref{ssFermat} it follows that the Fermat hypersurface is of split type.
\end{remark}

\begin{theorem}\label{thmSplit}
Let $k$ and $d$ be positive integers, such that $d\geq 2+\frac{2}{k+1}$. Let $c$ be an integer $1\leq c \leq k+1$. 
Let $X\subset \Ps^{2k+1}$ be a smooth hypersurface of degree $d$ containing two linear spaces  $\Pi_1,\Pi_2$ of dimension $k$, intersecting in dimension $k-c$ and of split type. 
Suppose $d=3$ and $c=3$ or $d=4$ and $c=2$. Then for all $\lambda\in \Q^*$ we have that $\NL([\Pi_1],[\Pi_2])\subsetneq \NL([\Pi_1]+\lambda[\Pi_2])_{\red}$ in a neighbourhood of $X$.
Moreover, if $d=3$ and $k\geq 5$ or $d=4$ and $k\geq 3$ then $\NL([\Pi_1]+\lambda[\Pi_2])$ is reducible in a neighborhood of $X$.
\end{theorem}
\begin{proof}
We do first the case $d=4,c=2$.
Without loss of generality we may assume that $X=V(f)$ with
\[ f=\sum_{i=0}^{1}\sum_{j=2}^{3} x_i x_j Q_{ij}+\sum_{j=k+3}^{2k+1} x_jP_j\]
and $Q_{ij}\in \C[x_0,\dots,x_{3}]$.

Consider now the $K3$-surface $\Sigma=V(g)\subset \Ps^3$  with
 \[g=\sum_{i=0}^{1}\sum_{j=2}^{3} x_i x_j Q_{ij}\]
 By construction $\Sigma$ is a quartic surface. Note that $\Sigma$ is also smooth: if $(z_0:z_1:z_2:z_3)$ were a singular point of $\Sigma$ then this would induce a singular point of $X$ along the non-empty set
 $V(\{x_iz_j-x_jz_i\}_{0\leq i<j\leq 3})\cap V(x_{k+3},\dots,x_{2k+1},P_{k+3},\dots,P_{2k+1})$.
 By construction, $\Sigma$ contains the two disjoint lines $V(x_0,x_1)$ and $V(x_2,x_3)$. Denote the corresponding classes in cohomology with $\ell_1,\ell_2$. 
 Since the period map for polarized K3 surfaces is an isomorphism we obtain that $\NL(\ell_1,\ell_2)$ has codimension 2 in $\C[x_0,\dots,x_3]_4$, whereas $\NL(\ell_1+\lambda\ell_2)$ has codimension 1. 
 
 Hence we can pick a curve $B\subset \NL(\ell_1+\lambda \ell_2)$ such that $B\cap \NL(\ell_1,\ell_2)=\{[\Sigma]\}$ and such that the tangent direction of $B$ is not contained in the tangent space of $\NL(\ell_1,\ell_2)$. Denote with $0$ the point corresponding to $[\Sigma]$.
Let $(\Sigma_t)_{t\in B}$ be the corresponding family of quartic surfaces.

After replacing $B$ by a Zariski open subset containing $0$, we can pick  a family of projective curves $(C_t)_{t\in B}$ in $\Ps^3$ such that $C_t\subset \Sigma_t$ for all $t$, and $[C_t]_{\prim}\neq 0$. Such a family of curves exists since $\ell_1+\lambda \ell_2$ a Hodge class on $\Sigma_t$. In particular there exists a line bundle $\cL_t$ on $S_t$ 
such that $c_1(\cL_t)=\ell_1+\lambda\ell_2$. For $q\gg0$ we have that  
the line bundle $\cL_t \otimes \cO_{S_t}(q)$ is effective. Moreover, using semicontinuity, we can pick $q$ uniformly on an open subset containing $0$. Since $[c_1(\cL_t \otimes \cO_{S_t}(q))]_{\prim}=[c_1(\cL_t)]_{\prim}$ holds, we have that both classes yield the same Hodge locus.
 
 Let $g_t$ be a family of polynomials such that $\Sigma_t=V(g_t)$ and $g_0=g$. Let $Z_t\subset \Ps^{2k+1}$ be the cone over $C_t$. Let $Y_t=Z_t\cap V(x_{k+3},\dots,x_{2k+1})$. Then $\dim Z_t=2k-1$ and $\dim Y_t=\dim Z_t-(k-1)=k$. Moreover $[Y_0]_{\prim}=([\Pi_1]+\lambda[\Pi_2])_{\prim}$ in $H^{2k}(X,\Q)_{\prim}$. 
  Consider now 
 \[ f_t:=g_t+\sum_{j=k+3}^{2k+1} x_jP_j\]
 and $X_t=V(f_t)$. 
 In particular, $Y_t\subset X_t$ and therefore $X_t\in \NL([\Pi_1]+\lambda[\Pi_2])$ for all $t\in B$.
 To show that $X_t \not \in \NL([\Pi_1],[\Pi_2])$ for some $t$ in a neighborhood of $0$, note that by construction the tangent direction of $B$ at $t=0$ is not contained in the tangent space of $\NL([\Pi_1],[\Pi_2])$.

 For $d=c=3$ we can proceed similarly.
 Without loss of generality we may assume that $X=V(f)$ with
\[ f=\sum_{i=0}^{2}\sum_{j=3}^{5} x_i x_j Q_{ij}+\sum_{j=k+4}^{2k+1} x_jP_j\]
and $Q_{ij}\in \C[x_0,\dots,x_{3}]$.
Consider now the family of cubic fourfolds  $\Sigma=V(g)$, with
 \[g=\sum_{i=0}^{2}\sum_{j=3}^{5} x_i x_j Q_{ij}\]
  By construction $\Sigma$ is a cubic fourfold. Note that $\Sigma$ is also smooth, if $(z_0:z_1:z_2:z_3:z_4:z_5)$ were a singular point then this would induce a singular point of $X$ along the nonempty set
 $V(\{x_iz_j-x_jz_i\}_{0\leq i<j\leq 5})\cap V(x_{k+4},\dots,x_{2k+1},P_{k+4},\dots,P_{2k+1})$.
 
By construction $
\Sigma$ contains the two planes $V(x_0,x_1,x_2)$ and $V(x_3,x_4,x_5)$. Denote the corresponding classes in cohomology with $\pi_1,\pi_2$. 
 Now $\NL(\pi_1,\pi_2)$ has codimension 2 in $\C[x_0,\dots,x_5]_4$, whereas $\NL(\pi_1+\lambda \pi_2)$ has codimension 1 (since the period map is an isomorphism).
 
 Hence we can pick a curve $B\subset\NL(\pi_1+\lambda \pi_2)$ such that $B\cap \NL(\pi_1,\pi_2)=\{[\Sigma]\}$ and such that the tangent direction of $B$ is not contained in the tangent space of $\NL(P_1,P_2)$. Denote with $0$ the point corresponding to $[\Sigma]$.

  Let $(\Sigma_t)_{t\in B}$ be the corresponding family of cubics fourfolds.
 Then on $\Sigma_t$ we have that the class in $H^4(\Sigma_t,\Q)$ corresponding to $\pi_1+\lambda \pi_2$ is a Hodge class, i.e., there exists irreducible surfaces $Z_1,\dots,Z_s$, and integer $a_1,\dots,a_s$ such that this class equals $\sum a_i[Z_i]$.  However, if there is a nonzero polynomial of degree $d_i$ vanishing on $Z_i$ then $d_i[h^2]-[Z_i]$ is effective. In particular, there is a $q\gg 0$ such that $q[h^2]+\sum a_i[Z_i]$ is effective. Moreover, this $q$ can be chosen uniformly on an open in a  neighborhood of $0$. In particular, after replacing $B$ by a Zariski open subset containing $0$, we can pick  a family of  surfaces $(C_t)_{t\in B}$ such that $C_t\subset \Sigma_t$ for all $t$, and $[C_t]_{\prim}\neq 0$.
 
 Let $g_t$ be a family of polynomials such that $\Sigma_t=V(g_t)$ and $g_0=g$. Let $Z_t\subset \Ps^{2k+1}$ be the cone over $C_t$. Let $Y_t=Z_t\cap V(x_{k+4},\dots,x_{2k+1})$. Then $\dim Z_t=2k-2$ and $\dim Y_t=\dim Z_t-(k-2)=k$. Moreover $[Y_0]_{\prim}=([\Pi_1]+\lambda[\Pi_2])_{\prim}$ in $H^{2k}(X,\Q)_{\prim}$. 
  Consider now 
 \[ g_t:=P_t+\sum_{j=k+3}^{2k+1} x_jP_j\]
 and $X_t=V(f_t)$. 
 
 In particular, $Y_t\subset X_t$ and therefore $X_t\in \NL([\Pi_1]+\lambda[\Pi_2])$.
 To show that $X_t \not \in \NL([\Pi_1],[\Pi_2])$ for some $t$ in a neighborhood of $0$, note that by construction the tangent direction of $B$ at $t=0$ is not contained in the tangent space of $\NL([\Pi_1],[\Pi_2])$.
 
 To prove the final statement, note that if $d=4$ and $k\geq 3$ or $d=3$ and $k\geq 5$ then by Theorem~\ref{mainThm} there is  a Zariski open subset $U$ of $\NL([\Pi_1],[\Pi_2])$ such that  
 \[\dim \NL([\Pi_1],[\Pi_2])=\dim T_{X'}\NL([\Pi_1],[\Pi_2])=\dim T_{X'}\NL([\Pi_1]+\lambda[\Pi_2])\] for all $X'\in U$, except for all but one nonzero value of $\lambda$. If $\NL([\Pi_1]+\lambda [\Pi_2])$ were irreducible then  $\NL([\Pi_1],[\Pi_2])$ would have positive codimension in this locus,  contradicting the above equality.
\end{proof}

\section{Examples satisfying Movasati's conjecture}\label{secExaPos}
In the previous sections we showed that for $d=4,c=2,k\geq 3$ and for $c=d=3,k\geq 5$ that Movasati's conjecture does not hold for all but at most one nonzero choice of $\lambda$. In this section we will present in both cases an example of the single nonzero value of $\lambda$ for which Movasati's conjecture holds for all $k$. We start with the case $d=4,c=2$:

\begin{proposition} Let $k\geq 1$ be an integer. Let $X\subset \Ps^{2k+1}$ be a hypersurface of degree $4$ containing two $k$-planes intersecting in dimension 2.
Then there exists a flat family $(X_t,Y_t)$ where $X_t\subset \Ps^{2k+1}$ is  an irreducible quartic hypersurface, $Y_t\subset X_t$ is subscheme of degree 4 and of pure dimension $k$, such that $Y_t$ is irreducible for general $t$, $X_0=X$ and $[Y_0]_{\prim}=[\Pi_1]_{\prim}-[\Pi_2]_{\prim}$.
\end{proposition}

\begin{proof}
Without loss of generality we may assume that \[ \Pi_1=V(x_0,x_1,x_{k+3},\dots,x_{2k+1}) \mbox{  and }\Pi_2=V(x_2,x_3,x_{k+3},\dots,x_{2k+1}).\] 
Then $X$ is the zeroset of a polynomial $f_0$, with
\[ f_0:=x_0x_2 Q_{02}+x_1x_2Q_{12}+x_0x_3Q_{03}+x_1x_3Q_{13}+\sum_{i=k+3}^{2k+1} x_i P_i\]
where $Q_{ij}\in \C[x_0,\dots,x_{k+2}]_{2}$ and $P_i\in \C[x_0,\dots,x_{2k+1}]_3$.

Consider now $X_t$ the zero set of
\[f_t:=f_0+t(Q_{13}Q_{02}-Q_{12}Q_{03})\]
Note that for $t\neq 0$ we have that
\[ f_t\in I^{(t)}:= \langle x_0x_2+tQ_{13},x_1x_2-tQ_{03},x_{k+3},\dots,x_{2k+1}\rangle \]
since
\[ f_t=(x_0x_2+tQ_{13})(Q_{02}+\frac{1}{t}x_1x_3)+(x_1x_2-tQ_{03})(Q_{12}-\frac{1}{t}x_0x_3).\]
In particular, for $t\neq 0$ we have that $X_t$ contains $Y_t:=V(I^{(t)})$; that  $I(Y_t)=I^{(t)}$ and that $Y_t$ is a  complete intersection $Y_t$ of multidegree $(2,2,1,\dots,1)$. 

Note that  $t\neq 0$, this ideal contains
\[ \frac{1}{t} x_1(x_0x_2+tQ_{13})-x_0(x_1x_2-tQ_{03})=x_1Q_{13}+x_0Q_{03}.\]
Define now $I^{(0)}:=\langle x_0x_2,x_1x_2,x_1Q_{13}+x_0Q_{03}\rangle$.
Then each of the generator is a specialization to $t=0$ of an element for $I^{(t)}$. Moreover, one easily checks that the Hilbert functions of $I^{(0)}$ and $I^{(t)}$ coincide.  
Hence  $Y_t=V(I^{(t)})$ is a flat family.  Moreover,
\[ V(x_0x_2,x_1x_2,x_1Q_{13}+x_0Q_{03})=V(x_2,x_1Q_{13}+x_0Q_{03})\cup V(x_0,x_1)\]
Hence $Y_0$  is the union of $\Pi_1$ and $k$-dimensional subscheme whose class is $h^k-[\Pi_2]$

In particular, $[Y_0]_{\prim}=[\Pi_1]-[\Pi_2]$.
\end{proof}

\begin{corollary}\label{corMovQua} Suppose $d=4,c=2$, $k\geq 1$. Then $\NL([\Pi_1],[\Pi_2])$ is a subscheme of $\NL([\Pi_1]-[\Pi_2])$ of positive codimension.
\end{corollary}

\begin{remark} This two results contain also a proof for \cite[Conjecture 18.1]{MovBook} in the case $d=4$.
\end{remark}

The next example is for $d=c=3$. In this case we consider cubic hypersurfaces which contains an iterated cone over the Pl\"ucker embedding of $\G(2,5)$ in $\Ps^9$.

Let $\G(2,5)$ be the Grassmannian variety parametrizing two dimensional subspaces of $\C^5$. Consider its Pl\"ucker embedding in $\Ps^9$.
Let $p_{ij}$ be the $1\leq i<j\leq 5$ be the Pl\"ucker coordinates for $\G(2,5)$. We have the following 5 Pl\"ucker relations
 \begin{eqnarray*}
p_1&=&p_{12}p_{34}-p_{13}p_{24}+p_{14}p_{23}\\
p_2&=&p_{12}p_{35}-p_{13}p_{25}+p_{15}p_{23}\\
p_3&=&p_{12}p_{45}-p_{14}p_{25}+p_{15}p_{24}\\
p_4&=&p_{13}p_{45}-p_{14}p_{35}+p_{15}p_{34}\\
p_5&=&p_{23}p_{45}-p_{24}p_{35}+p_{25}p_{34}\\
\end{eqnarray*}
The image of $\G(2,5)$ under the Pl\"ucker embedding is $V(p_1,\dots,p_5)$, this is a sixfold, hence of codimension 3.

\begin{proposition} \label{prpMovCub} Let $k\geq 2$ be an integer.  Let $X\subset \Ps^{2k+1}$ be a cubic hypersurface containing two $k$-planes $\Pi_1,\Pi_2$ intersecting in dimension $k-3$, then there exists a flat family of degree $d$ 
hypersurfaces $(X_t,,Y_t)$ where $X_t$ is a cubic hypersurface  and $Y_t$ is a linear section of an iterated cone over the Pl\"ucker embedding of $\G(2,5)$, for $t\neq 0$, $X_0=X$ and $[Y_0]=[\Pi_1]+[ \Pi_2]+h^2$.
\end{proposition}

\begin{proof}
Without loss of generality we may assume that $\Pi_1:x_0=x_1=x_2=x_{k+4}=\dots=x_{2k+1}=0$ and that $\Pi_2: x_3=x_4=x_5=x_{k+4}=\dots=x_{2k+1}$.
Then $X=V(F)$, where
\[F= \sum_{i=0}^2\sum_{j=3}^5x_ix_j L_{ij} +\sum_{i=k+4}^{2k+1}x_j Q_j\]

Let $f_t$ be
\[ \sum_{i=0}^2\sum_{j=3}^5x_ix_j L_{ij} -t \det( L_{(i-1),(j+2)})+\sum_{i=k+4}^{2k+1}x_j Q_j \]
and $X_t=V(f_t)$. Then $X_0=X$.
Write $H=\sum_{i=k+4}^{2k+1}x_j Q_j$.

Now we can  regroup the monomials in $f_t$ such that $f_t-H$ equals
\[\begin{matrix}
x_4(x_0L_{04}+x_1L_{14}+x_2L_{24})+ x_5(x_0L_{05}+x_1L_{15}+x_2L_{25})+\dots\\+L_{03}(x_0x_3-tL_{14}L_{25}+tL_{15}L_{24})+L_{13}(x_1x_3+tL_{04}L_{25}-tL_{05}L_{24})+\dots\\
+L_{23}(x_2x_3-tL_{04}L_{15}+tL_{05}L_{14})\end{matrix}\]

A straight forward calculation shows that if for $t\neq0$ we substitute
\[ \begin{matrix}(p_{12},p_{13},p_{14},p_{15},p_{23},p_{24},p_{25},p_{34},p_{35},p_{45})\mapsto\\
(x_0,x_1,L_{24},L_{25},x_2,-L_{14},-L_{25},L_{04},L_{05},\frac{1}{t}x_3)\end{matrix}\]
in
\[ x_4p_1+x_5p_2+t(L_{03}p_3+L_{13}p_4+L_{23}p_5)\] 
then we obtain  $f_t-H$. In particular, there is a linear change of coordinates $\varphi:\Ps^{2k+1}\to \Ps^{2k+1}$ such that $\varphi(V(f_t))$ contains a $k$-dimensional cone over the Pl\"ucker embedding of $\G(2,5)$.

Let $Y_t$ be the corresponding $k$-dimensional subscheme of $X_t$. Then for $t\neq 0$ we have that  $I^{(t)}:=I(Y_t)=(p_1,p_2,p_3,tp_4,tp_5)$. A straightforward calculation shows that the latter ideal is generated by $p_1$, $p_2$ and
\[x_0x_3+t(L_{24}L_{15}-L_{14}L_{25}), x_1x_3-t(L_{04}L_{15}+L_{14}L_{05}),x_2x_3+t(L_{14}L_{05}-L_{04}L_{15}).\]
Similar to the previous proof, we set
\[I ^{(0)}=\langle p_1,p_2,x_0x_3,x_1x_3,x_2x_3\rangle.\]
This is the intersection of two complete intersection ideals
\[ \langle p_1,p_2,x_3\rangle  \mbox { and } \langle  x_0,x_1,x_2\rangle\]
and each of the generators is the specialization to $t=0$ of an element in $I ^{(t)}$. Moreover, one easily checks that the Hilbert function of $I^{(0)}$ coincide with the Hilbert function of $I^{(t)}$ for $t\neq 0$.
 In  particular, $(X_t,Y_t)$ is a flat family.

Note that
\[f_0=x_4p_1+x_5p_2+L_{03}x_0x_3+L_{13}x_1x_3+L_{23}x_2x_3\]
Let $\gamma_1,\gamma_2\in \CH^k(X_0)$ be the classes $\gamma_1=[V(x_3,p_1,p_2,x_{k+4},\dots,x_{2k+1})]$, $\gamma_2=[V(x_3,x_4,p_2,x_{k+4},\dots,x_{2k+1})]$. Then $2h^k=[V(x_3,p_2,x_{k+4},\dots,x_{2k+1})]=\gamma_1+\gamma_2$

Now $[\Pi_2]+\gamma_2=[V(x_3,x_4,x_{k+4},\dots,x_{2k+1})]=h^k$
Hence
\[ [\Pi_1]+[\Pi_2]=[\Pi_1]+h^k-\gamma_2=[\Pi_1]+\gamma_1-h^k\]
Now $[\lim_{t\to 0} Y_t]=[\Pi_1]+\gamma_1=[\Pi_1]+[\Pi_2]+h^k$.
\end{proof}

\begin{corollary}\label{corMovCub} Suppose $d=3,c=3$, $k\geq 2$. Then $\NL([\Pi_1],[\Pi_2])$ is a subscheme of $\NL([\Pi_1]+[\Pi_2])$ of positive codimension.
\end{corollary}

\begin{remark} This two results contain also a proof for \cite[Conjecture 18.1]{MovBook} for $d=3$, however, where we take $\check{r}=1$ rather than $-1$. If $\check{r}=-r$ then the  conjecture is false by Theorem~\ref{mainThm}.
\end{remark}

\begin{remark} 
If $L_{03}=0$ then there is a different degeneration. 

Consider $A_t$ the matrix
\[ \begin{pmatrix}
-L_{15} & L_{05}  &x_3 & x_4\\
x_0      & x_1 &   -tL_{24} &t L_{23}
\end{pmatrix}
\]
Then the ideal $I_t$ of $2\times 2$ minors of $A_t$ is generated by
\[ \begin{matrix} x_0x_3-tL_{24}L_{15},x_0x_4+tL_{15}L_{23}, x_1x_3+tL_{24}L_{05},\\x_1x_4-tL_{05}L_{23}, x_0L_{05}+x_1L_{15}, t(x_3L_{23}+x_4L_{24})\end{matrix}.\]
Denote these six minors withj $f_1,\dots,f_6$ respectively.
Consider now
\[ L_{03}f_1+L_{04}f_2+L_{13}f_3+L_{14}f_4+x_5f_5+x_2f_6\]
If we substitute $t=0$ then we recover our original $f$.

The limit for $I_t$ when $t\to 0$ is generated by 
\[ x_0x_3,x_0x_4, x_1x_3,x_1x_4, x_0L_{05}+x_1L_{15}, x_3L_{23}+x_4L_{24},\]
which is the intersection of 
\[ (x_0,x_1,  x_3L_{23}+x_4L_{24}) \mbox{ and } (x_3,x_4,x_0L_{05}+x_1L_{15})\]
The associated class $h-[\Pi_1]+h-[\Pi_2]$. Hence this corresponds to $a=b=1$.

In this case we have a degeneration of  a quartic subscheme, whereas in the above proof we used a quintic subscheme. If we work with cubic fourfolds, i.e., $k=2$, then we can always find a change of coordinates such that $L_{25}=0$. In this case it is well-known that a general threefold on $\cD_8$ (notation from \cite{Has, HasArt}) contains a quartic and a quintic scroll and that they  degenerate to a cubic fourfold with two disjoint planes. However, if $k\geq 5$ then we may not find a coordinate change such that $L_{25}$ vanishes. For this reason we need  to use the quintic degeneration in order to show that $\NL([\Pi_1],[\Pi_2])$ is a closed subscheme of $\NL([\Pi_1]+[\Pi_2])$ if $k$ is large.
\end{remark}

\appendix
\section{Computer code used}\label{secCode}
We calculated the rank of the pairing use singular. We include our scripts to calculate the rank of the pairing in the 3 examples \ref{ExaConA}, \ref{ExaConB}, \ref{ExaConC}.
The smoothness of the example can be checked by two lines and is therefore not included.

For Example~\ref{ExaConA}:
\begin{scriptsize}
\begin{verbatim}
ring R=0,(x2,x3,x4,x5),dp;
poly Q02,Q13;
poly P6,P7;
//We are working in P^7, however I1 contains x0,x1,x6,x7. 
//Hence we can do all calculations modulo these variables.

Q02=x2^2+x4^2; 
Q13=x3^2+x5^2; 
P6=x4^3;x
P7=x5^3;

ideal i1=x2*Q02,x3*Q13,P6,P7;
ideal isum=x2,x3,P6,P7;
i1=std(i1);
isum=std(isum);

ideal id4=kbase(i1,4);
print("Dimension of (S/I1)_8:     "+string(size(kbase(i1,8))));
print("Dimension of (S/I1)_4:     "+string(size(id4)));
print("Dimension of (I1+I2/I1)_4: "+string(size(id4)-size(kbase(isum,4))));

list bas; // we are going to produce generating set for ((I1+I2)/I1)_4. 
int i,j;
poly pw;

for(i=1;i<=size(id4);i++)
{
  pw=id4[i]-NF(id4[i],isum);
  if(pw<>0)
  {
    bas=insert(bas,pw);
  }   
}

int n=size(bas);
matrix ma[n][n];

print("Nonzero entries of the Gram matrix:");
for(i=1;i<=n;i++)
{
  for(j=i;j<=n;j++)
  {
    pw=NF(bas[i]*bas[j],i1); // calculate the pairing
    if(pw<>0)
    {
      ma[i,j]=pw;
      ma[j,i]=pw;
      print([i,j]); // print the position of the nonzero entries
     }
  }
}

print("The determinant of the Gram matrix:"+string(det(ma)));

\end{verbatim}
\end{scriptsize}

For Example~\ref{ExaConB} the code is very similar:
\begin{scriptsize}
\begin{verbatim}
ring R=0,(x3,x4,x5,x6,x7,x8),dp;
poly Q03,Q14,Q25;
poly P9,P10,P11;
int i,j;
Q03=x3+x6;
Q14=x4+x7;
Q25=x5+x8;
P9=x6^2;
P10=x7^2;
P11=x8^2;

ideal i1=x3*Q03,x4*Q14,x5*Q25,P9,P10,P11;
ideal isum=x3,x4,x5,P9,P10,P11;
i1=std(i1);
isum=std(isum);

ideal id3=kbase(i1,3);
print("Dimension of (S/I1)_6:     "+string(size(kbase(i1,6))));
print("Dimension of (S/I1)_3:     "+string(size(id3)));
print("Dimension of (I1+I2/I1)_3: "+string(size(id3)-size(kbase(isum,3))));

list bas;
poly pw;
for(i=1;i<=size(id3);i++)
{
  pw=id3[i]-NF(id3[i],isum);
  if(pw<>0)
  {
    bas=insert(bas,pw);
  }
}

int n=size(bas);

print("Nonzero entries of the Gram matrix:");
matrix ma[n][n];

for(i=1;i<=n;i++)
{
  for(j=i;j<=n;j++)
  {
    pw=NF(bas[i]*bas[j],i1);
    if(pw<>0)
    {
      ma[i,j]=pw;
      ma[j,i]=pw;
      print([i,j]);
    }
  }
}

print("The determinant of the Gram matrix:"+string(det(ma)));
\end{verbatim}
\end{scriptsize}

For Example~\ref{ExaConC}:
\begin{scriptsize}
\begin{verbatim}
ring R=0,(x2,x3,x4,x5,x6,x7),dp;
poly Q02,Q03,Q12,Q13;
poly P8,P9,P10,P11;
int i,j;
list bas; // generating sets for (I_1+I_2/I_1) in deg 3.
poly pw; // polynomial under consideration
Q02=x2+x4+x6;
Q03=3*x3+x5;
Q12=x6+5*x7;
Q13=x3+x5+x7; 
P8=x4^2;
P9=x5^2;
P10=x6^2;
P11=x7^2;
ideal i1=x2*Q02+x3*Q03,x2*Q12+x3*Q13,P8,P9,P10,P11; //I1 mod x0,x1,x8,x9,x10,x11
ideal isum=x2,x3,P8,P9,P10,P11; //I1+I2 mod x0,x1,x8,x9,x10,x11
i1=std(i1);
isum=std(isum);

ideal id3=kbase(i1,3); // basis for (S/I1)_3 

print("Dimension of (S/I1)_6:     "+string(size(kbase(i1,6))));
print("Dimension of (S/I1)_3:     "+string(size(id3)));
print("Dimension of (I1+I2/I1)_3: "+string(size(id3)-size(kbase(isum,3))));


for(i=1;i<=size(id3);i++)
{
  pw=id3[i]-NF(id3[i],isum);
  if(pw<>0)
  {
    bas=insert(bas,pw);
  }
}

int n;
n=size(bas);

matrix ma[n][n]; 

//print("Nonzero entries for the Gram matrix for degree 3:");
for(i=1;i<=n;i++)
{
  for(j=i;j<=n;j++)
  {
    pw=NF(bas[i]*bas[j],i1);
    if(pw<>0)
    {
      ma[i,j]=pw;
      ma[j,i]=pw;
//      print([i,j]);

    } 
  }
}
print("The determinant of the Gram matrix: "+string(det(ma)));
\end{verbatim}

\end{scriptsize}
\bibliographystyle{plain}
\bibliography{remke2}

\end{document}